\documentclass[11pt]{scrartcl}
\usepackage[utf8]{inputenc}
\usepackage[T1]{fontenc}
\usepackage{lmodern}
\usepackage{alphabeta}

\usepackage{titling}
\usepackage{xspace}

\usepackage[a4paper, top=3cm, bottom=2cm, left=2.5cm, right=2.5cm, includefoot]{geometry}

\usepackage[inline]{enumitem}

\usepackage{cite}
\usepackage{dsfont}
\usepackage{setspace}

\usepackage{bm}

\usepackage{todonotes}
\usepackage{amssymb}
\usepackage{microtype}
\usepackage{amsmath}
\usepackage{amssymb}
\usepackage{amsfonts}
\usepackage{mathtools}
\usepackage[mathscr]{euscript}

\usepackage[hyphens]{url}
\usepackage{amsthm}

\usepackage{tikz}
\usepackage{xcolor}
\usepackage{subcaption}

\usepackage{hyperref}
\usepackage{cleveref}

\usepackage{nicefrac}

\usepackage{thm-restate}
\usepackage{comment}

\colorlet{myGreen}{green!50!black}
\colorlet{myLightgreen}{green}
\colorlet{myRed}{red!90!black}
\definecolor{myBlue}{rgb}{0.25, 0.0, 1.0}
\definecolor{myLightBlue}{rgb}{0.39, 0.58, 0.93}
\colorlet{myViolet}{myBlue!55!myRed}
\definecolor{myOrange}{rgb}{1.0, 0.66, 0.07}

\definecolor{CornflowerBlue}{rgb}{0.39, 0.58, 0.93}
\definecolor{DarkGoldenrod}{rgb}{0.72, 0.53, 0.04}
\definecolor{BritishRacingGreen}{rgb}{0.0, 0.26, 0.15}
\definecolor{DarkMagenta}{rgb}{0.55, 0.0, 0.55}
\definecolor{AO}{rgb}{0.0, 0.5, 0.0}
\definecolor{BostonUniversityRed}{rgb}{0.8, 0.0, 0.0}
\definecolor{myRed}{rgb}{0.8, 0.0, 0.0}
\definecolor{DarkMidnightBlue}{rgb}{0.0, 0.2, 0.4}
\definecolor{DarkTangerine}{rgb}{1.0, 0.66, 0.07}
\definecolor{AppleGreen}{rgb}{0.55, 0.71, 0.0}
\definecolor{BrightUbe}{rgb}{0.82, 0.62, 0.91}
\definecolor{Amethyst}{rgb}{0.6, 0.4, 0.8}
\definecolor{DarkGray}{rgb}{0.52, 0.52, 0.51}
\definecolor{Gray}{rgb}{0.66, 0.66, 0.66}
\definecolor{BananaYellow}{rgb}{1.0, 0.88, 0.21}
\definecolor{Amber}{rgb}{1.0, 0.75, 0.0}
\definecolor{LightGray}{rgb}{0.83, 0.83, 0.83}
\definecolor{PrincetonOrange}{rgb}{1.0, 0.56, 0.0}
\definecolor{DeepCarrotOrange}{rgb}{0.91, 0.41, 0.17}
\definecolor{MidnightBlue}{rgb}{0.1, 0.1, 0.44}
\definecolor{HotMagenta}{rgb}{1.0, 0.11, 0.81}
\definecolor{Iceberg}{rgb}{0.44, 0.65, 0.82}
\definecolor{DarkCyan}{rgb}{0.0, 0.55, 0.55}

\vfuzz \hfuzz
\setlist[itemize]{topsep=0pt,partopsep=0pt,itemsep=0pt,parsep=0pt}
\setlist[itemize,1]{label={\small\textbullet}}
\setlist[itemize,2]{label={\tiny\textbullet}}
\setlist[itemize,3]{label=$\cdot$}
\setlist[enumerate]{topsep=0pt,partopsep=0pt,itemsep=0pt,parsep=0pt}
\setlist[enumerate,1]{label=\roman*)}
\setlist[enumerate,2]{label=\alph*)}
\setlist[enumerate,3]{label=\arabic*)}

\hypersetup{
colorlinks=true,
linkcolor=AO!65!black,
citecolor=AO!65!black,
urlcolor=AppleGreen!65!black,
bookmarksopen=true,
bookmarksnumbered,
bookmarksopenlevel=2,
bookmarksdepth=3
}

\theoremstyle{definition}

\newtheorem{environment}{Environment}[section]

\newtheorem{lemma}[environment]{Lemma}
\crefname{lemma}{lemma}{lemmata}
\crefformat{lemma}{#2Lemma~#1#3}
\Crefformat{lemma}{#2Lemma~#1#3}

\newtheorem*{lemma*}{Lemma}
\crefname{lemma*}{lemma}{lemmata}
\crefformat{lemma*}{#2Lemma~#1#3}
\Crefformat{lemma*}{#2Lemma~#1#3}

\crefname{proposition}{proposition}{propositions}
\crefformat{proposition}{#2Proposition~#1#3}
\Crefformat{proposition}{#2Proposition~#1#3}

\newtheorem{corollary}[environment]{Corollary}
\crefname{corollary}{corollary}{corollaries}
\crefformat{corollary}{#2Corollary~#1#3}
\Crefformat{corollary}{#2Corollary~#1#3}

\newtheorem{theorem}[environment]{Theorem}
\crefname{theorem}{theorem}{Theorems}
\crefformat{theorem}{#2Theorem~#1#3}
\Crefformat{theorem}{#2Theorem~#1#3}

\newtheorem{conjecture}[environment]{Conjecture}
\crefname{conjecture}{conjecture}{Conjectures}
\crefformat{conjecture}{#2Conjecture~#1#3}
\Crefformat{conjecture}{#2Conjecture~#1#3}

\newtheorem*{hypothesis*}{Hypothesis}
\crefname{hypothesis*}{conjecture}{Conjectures}
\crefformat{hypothesis*}{#2Conjecture~#1#3}
\Crefformat{hypothesis*}{#2Conjecture~#1#3}

\newtheorem{observation}[environment]{Observation}
\crefname{observation}{observation}{Observations}
\crefformat{observation}{#2Observation~#1#3}
\Crefformat{observation}{#2Observation~#1#3}

\crefname{example}{example}{examples}
\crefformat{example}{#2Example~#1#3}
\Crefformat{example}{#2Example~#1#3}

\crefname{remark}{remark}{remarks}
\crefformat{remark}{#2Remark~#1#3}
\Crefformat{remark}{#2Remark~#1#3}

\crefname{figure}{figure}{figures}
\crefformat{figure}{#2Figure~#1#3}
\Crefformat{figure}{#2Figure~#1#3}

\crefname{equation}{equation}{Equations}
\crefformat{equation}{#2#1#3}
\Crefformat{equation}{#2#1#3}

\crefname{chapter}{chapter}{chapters}
\crefformat{chapter}{#2Chapter~#1#3}
\Crefformat{chapter}{#2Chapter~#1#3}

\crefname{section}{section}{sections}
\crefformat{section}{#2Section~#1#3}
\Crefformat{section}{#2Section~#1#3}

\crefname{algorithm}{algorithm}{algorithms}
\crefformat{algorithm}{#2Algorithm~#1#3}
\Crefformat{algorithm}{#2Algorithm~#1#3}

\crefname{notation}{notation}{notations}
\crefformat{notation}{#2Notation~#1#3}
\Crefformat{notation}{#2Notation~#1#3}

\newtheorem{question}[environment]{Question}
\crefname{question}{question}{questions}
\crefformat{question}{#2Question~#1#3}
\Crefformat{question}{#2Question~#1#3}

\crefname{problem}{problem}{problem}
\crefformat{problem}{#2problem~#1#3}
\Crefformat{problem}{#2problem~#1#3}

\crefname{claim}{claim}{claims}
\crefformat{claim}{#2Claim~#1#3}
\Crefformat{claim}{#2Claim~#1#3}

\crefname{definition}{definition}{definitions}
\crefformat{definition}{#2Definition~#1#3}
\Crefformat{definition}{#2Definition~#1#3}

\usetikzlibrary{calc}
\usetikzlibrary{fit}
\usetikzlibrary{decorations}
\usetikzlibrary{decorations.pathmorphing}
\usetikzlibrary{decorations.text}
\usetikzlibrary{external}
\usetikzlibrary{shapes,hobby}

\tikzset{
	position/.style args={#1:#2 from #3}{
		at=($(#3)+(#1:#2)$)
	}
}

\tikzset{
  v:main/.style = {draw, circle, scale=0.8, thick,fill=black,inner sep=0.7mm},
    v:maingray/.style = {draw, circle, scale=0.65, thick,color=gray,fill=gray,inner sep=0.7mm},
  v:mainempty/.style = {draw, circle, scale=0.8, thick,fill=white,inner sep=0.7mm},
  v:mainellipse/.style = {draw, ellipse, scale=0.8, thick,fill=white,inner sep=0.7mm},
    v:mainbox/.style = {draw, scale=0.8, thick,fill=white,inner sep=0.7mm},
    v:mainred/.style = {draw, circle, scale=0.65, thick,fill=red,inner sep=0.7mm},
        v:maingreen/.style = {draw, circle, scale=0.65, thick,fill=myGreen,inner sep=0.7mm},
                v:mainemptygreen/.style = {draw, circle, scale=0.65, thick,color=myGreen,fill=white,inner sep=0.7mm},
  v:mainemptygray/.style = {draw, circle, scale=0.65, thick,color=gray,fill=white,inner sep=0.7mm},
  v:tinytree/.style = {draw, circle, scale=0.03, thick,fill=black},
  v:middle/.style = {draw, circle, scale=0.3,thick,fill=Gray,color=Gray,inner sep=1mm},
  v:border/.style = {draw, circle, scale=0.75, thick,minimum size=10.5mm},
  v:mainfull/.style = {draw, circle, scale=1, thick,fill},
  v:ghost/.style = {inner sep=0pt,scale=1},
  v:marked/.style = {circle, scale=1.3, fill=DarkGoldenrod,opacity=0.4},
  v:tree/.style = {draw, circle, scale=0.45, thick,fill=black},
  >={latex},
  e:shiftedright/.style = {decoration={sl, raise=0.65pt},  decorate},
  e:shiftedleft/.style  = {decoration={sl, raise=-0.65pt}, decorate},
  e:marker/.style = {line width=8.5pt,line cap=round,opacity=0.35,color=DarkGoldenrod},
  e:colored/.style = {line width=1.8pt,color=BostonUniversityRed,cap=round,opacity=0.8},
  e:coloredthin/.style = {line width=1.6pt,opacity=1},
  e:coloredborder/.style = {line width=3.4pt},
  e:main/.style = {line width=1pt},
  e:thick/.style = {line width=2pt},
  e:mainplus/.style = {line width=1.15pt},
  e:mainthin/.style = {line width=0.6pt},
  e:extra/.style = {line width=1.3pt,color=LavenderGray},
  e:matching1/.style = {line width=2.2pt, color=myGreen, cap=round},
  e:matching2/.style = {line width=1.7pt, color=myRed, dashed, cap=round},
  e:matching2border/.style = {line width=0.35pt, color=white, dashed, cap=round, double=myRed, double distance=1.7pt},
  e:positive/.style = {line width=1.35pt},
  e:negative/.style = {line width=1.35pt,densely dotted},
}

\DeclareMathOperator{\atd}{\alpha\mbox{-}\mathsf{td}}

\title{Excluding an induced wheel minor in graphs\\ without large induced stars\thanks{An extended abstract of this work was accepted for the proceedings of the 51st International Workshop on Graph-Theoretic Concepts in Computer Science (WG 2025).}}
\predate{}
\date{}
\postdate{}

\preauthor{}
\DeclareRobustCommand{\authorthing}{
	\begin{center}
	    Mujin Choi\thanks{Supported by the Institute for Basic Science (IBS-R029-C1)} \\
        {\small Department of Mathematical Sciences, KAIST, Daejeon, Korea}\\
		{\small Discrete Mathematics Group, Institute for Basic Science, Daejeon, Korea.}\\
		\href{mailto:mujinchoi@kaist.ac.kr}{mujinchoi@kaist.ac.kr}\\

        \bigskip
	    Claire Hilaire\thanks{Supported in part by the Slovenian Research and Innovation Agency (J1-4008 and N1-0370)}\\
		{\small FAMNIT, University of Primorska, Koper, Slovenia.\\
        \href{mailto:claire.hilaire@famnit.upr.si}{claire.hilaire@famnit.upr.si}}\\

        \bigskip
	    Martin Milanič\thanks{Supported in part by the Slovenian Research and Innovation Agency (I0-0035, research program P1-0285 and research projects J1-3003, J1-4008, J1-4084, J1-60012, and N1-0370) and by the research program CogniCom (0013103) at the University of Primorska.}\\
		{\small FAMNIT and IAM, University of Primorska, Koper, Slovenia.\\
	    \href{mailto:martin.milanic@upr.si}{martin.milanic@upr.si}}\\
	    
		\bigskip
	    Sebastian Wiederrecht\thanks{Sebastian Wiederrecht's research was partially supported by the Institute for Basic Science (IBS-R029-C1).} \\
		{\small School of Computing, KAIST, Daejeon, Korea.\\
		\href{mailto:wiederrecht@kaist.ac.kr}{wiederrecht@kaist.ac.kr}}
\end{center}}
\author{\authorthing}
\postauthor{}

\begin{document}
\maketitle
\thispagestyle{empty}

\begin{abstract}
We study a conjecture due to Dallard, Krnc, Kwon, Milanič, Munaro, Štorgel, and Wiederrecht stating that for any positive integer $d$ and any planar graph $H$, the class of all $K_{1,d}$-free graphs without $H$ as an induced minor has bounded tree-independence number.
A $k$-wheel is the graph obtained from a cycle of length $k$ by adding a vertex adjacent to all vertices of the cycle.
We show that the conjecture of Dallard et al.~is true when $H$ is a $k$-wheel for any $k\geq 3$.
Our proof uses a generalization of the concept of brambles to tree-independence number.
As a consequence of our main result, several important $\mathsf{NP}$-hard problems such as \textsc{Maximum Independent Set} are tractable on $K_{1,d}$-free graphs without large induced wheel minors.
Moreover, for fixed $d$ and $k$, we provide a polynomial-time algorithm that, given a $K_{1,d}$-free graph $G$ as input, finds an induced minor model of a $k$-wheel in $G$ if one exists.
\end{abstract}

\section{Introduction}
Treewidth is one of the most widely studied graph parameters based on so called \textsl{tree-decompositions}\footnote{We postpone the definition of tree-decompositions and treewidth to \Cref{sec:Prelim}.}.
The popularity of treewidth in theoretical computer science is due to its powerful properties for the design of (parameterized) algorithms.
As a prominent example of these properties, Courcelle's theorem (see, e.g., \cite{COURCELLE199012}) states that any decision problem expressible in \textsl{counting monadic second-order logic}, written $\mathsf{CMSO}_2$ (see \cite{courcelle2012graph} for the definition), can be solved in linear time when the input is restricted to graphs whose treewidth is bounded by some constant.
On the structural graph theory side, treewidth plays a central role in the Graph Minor Theory by Robertson and Seymour (see, e.g., \cite{GraphMinor4,GraphMinor5}).

Unfortunately, graph classes of bounded treewidth are rather restricted.
This is because treewidth is a \textsl{minor-monotone} parameter and thus graphs of bounded treewidth are forced to be sparse: for every graph $G$ of treewidth at most $k$ we have that $|E(G)|\le k|V(G)|$.
Hence, bounding the treewidth naturally fails to consider several important but very simple graph classes.
There have been various attempts to deal with this problem by defining new width parameters that can be used to understand dense graph classes more thoroughly. 
Examples of such parameters are clique-width \cite{COURCELLE200077} (and the related rank-width \cite{OUM2006514}) but also more recent parameters such as twin-width \cite{BonnetTwidth1}.
Another possible approach to the design of new width parameters is to change the way the width of a tree-decomposition is evaluated to account for dense but still naturally tree-like graphs like chordal graphs.
In this context, Yolov \cite{YolovMinorMatchingHypertreeWidth}, and independently, Dallard, Milanič, and Štorgel \cite{dallard2021treewidthversuscliquenumber1} defined \textsl{tree-independence number}, denoted\footnote{The authors of \cite{dallard2021treewidthversuscliquenumber1} denoted the tree-independence number by $\mathsf{tree}$-$\alpha$.} by $\alpha$-$\mathsf{tw}$.
This new parameter measures the independence number of each bag in a tree-decomposition instead of its size.
While graphs of treewidth at most~$1$ are exactly the forests, and thus treewidth can be considered to measure how close a given graph is to being a forest, graphs of tree-independence number at most~$1$ are precisely the chordal graphs; hence, tree-independence number can be understood as a measure for how close a given graph is to being chordal.

Tree-independence number has gotten a lot of attention recently, for both structural and algorithmic reasons.
It appears to be a natural parameter for the study of \textsl{induced-minor-closed} graph classes~\cite{dallard2024treewidthversuscliquenumber2,hilaire2024treewidth,chudnovsky2024tree}. 
We say that a graph $H$ is an \emph{induced minor} of a graph $G$ if $H$ can be obtained from $G$ by vertex deletions and edge contractions.
The induced minor relation has been under consideration for quite some time for both its structural properties and algorithmic implications  (see~\cite{Matousek1988Polynomial,Fellows1995Complexity,Blasiok2019Induced,Korhonen2024Induced,hilaire2024treewidth}).
In particular, tree-independence number is monotone under the induced minor relation (see~\cite{dallard2024treewidthversuscliquenumber2}).

From the algorithmic point of view, in \cite{YolovMinorMatchingHypertreeWidth} (see also~\cite{dallard2024treewidthversuscliquenumber2}), the author showed that the well-known \textsc{Maximum Weight Independent Set} (\textsc{MWIS}) problem can be solved in polynomial time in graph classes with bounded tree-independence number.
In fact, in \cite{dallard2024computingtreedecompositionssmall}, Dallard, Fomin, Golovach, Korhonen, and Milanič showed that tree-independence number is the most general parameter among similar tree-decomposition based parameters that implies tractability of MWIS (unless $\mathsf{P} = \mathsf{NP}$).
Bounded tree-independence number implies polynomial-time solvability of a number of algorithmic problems other than \textsc{MWIS} (see~\cite{dallard2024treewidthversuscliquenumber2,YolovMinorMatchingHypertreeWidth,Lima2024Tree}).
In particular, Lima, Milani{\v{c}}, Mur{\v{s}}i{\v{c}}, Okrasa, Rz{\k{a}}{\.z}ewski, and {\v{S}}torgel~\cite{Lima2024Tree} proved an \textsl{algorithmic metatheorem}, giving a generic algorithm that finds a maximum-weight induced subgraph of bounded clique number satisfying a fixed property definable in $\mathsf{CMSO}_2$.
As we will see in \Cref{sec:inducedMinorChecking}, this metatheorem allows to efficiently solve \textsc{$H$-Induced Minor Containment} in graphs of bounded tree-independence number (a similar result was recently obtained by Bousquet, Dallard, Dumas, Hilaire, Milanič, Perez, and Trotignon~\cite{Bousquet2025Induced}).
Note that \hbox{\textsc{$H$-Induced Minor Containment}} is known to be $\mathsf{NP}$-complete for some graphs $H$ (see~\cite{Fellows1995Complexity}), even if $H$ is a tree (see~\cite{Korhonen2024Induced}).


In this context, and from a more structural point of view, it is interesting to characterize graph classes with bounded tree-independence number, in particular in terms of excluded induced minors.
Recently, Dallard et al.~\cite{dallard2024treewidthversuscliquenumber4}
conjectured that for every positive integer $d$ and every planar graph $H$, every graph that excludes $K_{1,d}$ as an induced subgraph and $H$ as an induced minor has bounded tree-independence number.
Here, $K_{1,d}$ denotes the \emph{star} with $d$ leaves, i.e.\@ the complete bipartite graph where one side has exactly one vertex and the other side has $d$ vertices.
Alternatively, $K_{1,d}$-free\footnote{Given a graph $H$, a graph $G$ is \emph{$H$-free} if it does not contain $H$ as an induced subgraph.} graphs are exactly those where the maximum size of an independent set in the neighbourhood of any vertex is at most $d-1$.
The conjecture of Dallard et al.~can be seen as an induced minor version of the celebrated Grid Theorem by Robertson and Seymour \cite{GraphMinor5}.
Moreover, the conjecture generalises a related theorem by Korhonen \cite{Korhonen2023InducedGrid}, stating that for every positive integer $d$, every graph with maximum degree at most $d$ that excludes some fixed planar graph as an induced minor has bounded treewidth.

\subsection{Our results}

We prove a first result that shows that for every integer $k\ge 4$, there exist a $3$-connected planar graph $H$ on $k$ vertices such that excluding $H$ as an induced minor in $K_{1,d}$-free graphs yields bounded tree-independence number.

\begin{figure}
    \centering
    \begin{tikzpicture}[scale=0.37]
    \pgfdeclarelayer{background}
    \pgfdeclarelayer{foreground}
    \pgfsetlayers{background,main,foreground}
    \begin{pgfonlayer}{background}
        \pgftext{\includegraphics[width=10cm]{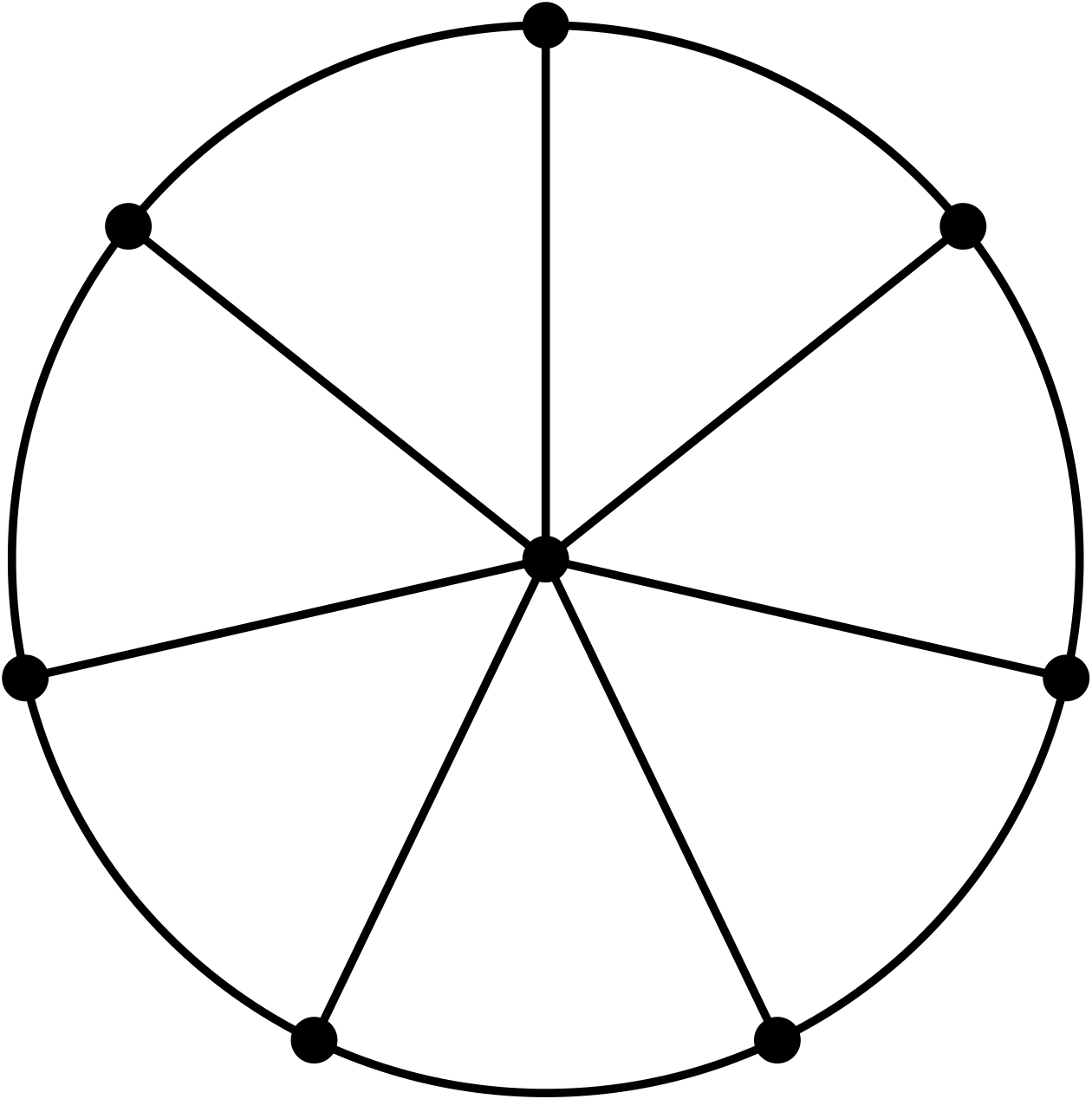}} at (C.center);
    \end{pgfonlayer}{background}
    \begin{pgfonlayer}{main}
    \end{pgfonlayer}{main}
    \begin{pgfonlayer}{foreground}
    \end{pgfonlayer}{foreground}
\end{tikzpicture}
    \caption{The $7$-wheel $W_7$}
    \label{fig:wheel}
\end{figure}

Let $\ell\geq 3$ be an integer.
The \emph{$\ell$-wheel}, denoted by $W_{\ell}$, is the graph obtained from the cycle on $\ell$ vertices by introducing one additional vertex adjacent to all vertices of the cycle (see \cref{fig:wheel} for an example).
Before this work, it was known that graphs excluding $W_{3}$ or $W_4$ as an induced minor have bounded tree-independence number \cite{dallard2024treewidth}.
In the same paper it was proven that for any $\ell\geq 5$, excluding $W_{\ell}$ alone as an induced minor is not enough to bound the tree-independence number.

The main result of this paper is, that in $K_{1,d}$-free graphs, excluding $W_{\ell}$ suffices to bound the tree-independence number for all integers $d\ge 1$ and $\ell\geq 3$.

\begin{restatable}{theorem}{thmforbiddingthewheel}\label{thm:forbiddingthewheel}
    There exists a function $f_{\ref{thm:forbiddingthewheel}}(d,\ell )\in\mathcal{O}\big(d\ell (\ell^{10}+2^{(\max\{ d+2,\ell \})^5})\big)$ such that for all integers $d\ge 1$ and $\ell\geq 3$ and every $K_{1,d}$-free graph $G$, either $G$ contains $W_{\ell}$ as an induced minor, or $\alpha\text{-}\mathsf{tw}(G)\leq f_{\ref{thm:forbiddingthewheel}}(d,\ell)$.
\end{restatable}

Our proofs are constructive, in the sense that they can be transformed into an algorithm that provides, given a $K_{1,d}$-free graph $G$ and an integer $\ell$ as input, either an induced minor model of $W_{\ell}$ in $G$, or a tree-decomposition where every bag has independence number at most $f_{\ref{thm:forbiddingthewheel}}(d,\ell)$ in time $|V(G)|^{g(d,\ell)}$, where $g$ is a non-decreasing computable function.
As a consequence, we obtain an algorithm that finds an induced $W_{\ell}$-minor, if there is one, in a $K_{1,d}$-free graph in polynomial time (see \Cref{thm:wheelChecking}).
In particular, this result identifies islands of tractability of the \hbox{\textsc{$H$-Induced Minor Containment}} problem even if the tree-independence number of the input graph is unbounded.

\medskip
Besides our main theorem, we also include a qualitative strengthening of a result from \cite{dallard2024treewidthversuscliquenumber4} stating that for any two positive integers $d$ and $k$, the class of $\{ K_{1,d},P_k\}$-free graphs has bounded tree-independence number.
(We denote by $P_k$ the $k$-vertex path.)
To this end, we introduce an independence version of treedepth.
Treedepth~\cite{NesetrilOdM2006Treedepth}, another widely studied graph parameter that measures how similar a graph is to a star.
One way to define treedepth is through the use of so-called \textsl{elimination forests}.
These are, roughly speaking, rooted forests on the same vertex set as the to-be-decomposed graph $G$, containing $G$ as a subgraph in their transitive closure.
The goal is to minimize the \textsl{depth} of such an elimination forest, which is standardly measured by taking the maximum number of vertices on a root-to-leaf path.
In $\alpha$-treedepth, we instead consider the maximum independence number of a subgraph of $G$ induced by a set of vertices on a root-to-leaf path of the elimination forest.
By doing so, we can allow large cliques for graph classes with bounded $\alpha$-treedepth.
We show that the bound of \cite{dallard2024treewidthversuscliquenumber4} for $\{ K_{1,d},P_k\}$-free graphs still applies when one replaces tree-independence number with $\alpha$-treedepth.
Assuming that an induced star is excluded, this is best possible, as the $\alpha$-treedepth of $P_k$ itself depends on $k$.
This also implies that an analogue of \cref{thm:forbiddingthewheel} where tree-independence number is replaced with $\alpha$-treedepth does not hold.

\begin{restatable}{theorem}{thmalphatreedepth}\label{thm:alphatreedepth}
    For every $k,d\in\mathbb{N}$ the following two statements are true.
    \begin{enumerate}
    	\item $\atd(P_k)=\lceil \log   (\frac{k}{3}+1)\rceil\geq \log (k/3)$, and
    	\item for every $\{P_k,K_{1,d} \}$-free graph $G$ it holds that $\atd(G)\leq \max\{1,(d-1)(k-2)\}$.
    \end{enumerate}
\end{restatable}

\subsection{Our method: A new tool for the study of tree-independence number}

One strength of treewidth is that it has many useful dual notions, such as brambles, well-linked sets, and tangles.
These have all been key concepts in proofs for the Grid Theorem \cite{Robertson1994Quickly,Reed2012Polynomial,Chuzhoy2021Towards}.
As a key tool towards our main result, we provide similar dual notions for the tree-independence number.
A similar approach was originally studied by Isolde Adler in her doctoral dissertation \cite{IsoldeThesis} for the more general notions of $f$-hypertree-width and the $f$-bramble number, where $f$ is any width function on a graph.
In this paper, we focus on the independence number to establish the following approximate duality between tree-independence number and a relevant notion for brambles.
We postpone the precise definitions to \Cref{sec:duality}. 
(For readers familiar with the notion of brambles: a \emph{strong bramble} is a bramble where every pair of elements have a non-empty intersection, and the \emph{$\alpha$-order} of a bramble $\mathcal{B}$ is the smallest independence number of a cover for~$\mathcal{B}$.)

\begin{restatable}{theorem}{thmbrambles}\label{thm:brambles}
Let $G$ be a graph.
Then, the following statements hold for every positive integer~$k$.
\begin{enumerate}
    \item If $G$ has a strong bramble of $\alpha$-order $k$, then $\alpha\text{-}\mathsf{tw}(G)\geq k$.
    \item If $\alpha\text{-}\mathsf{tw}(G)\geq 4k-2$, then $G$ has a strong bramble of $\alpha$-order at least $k$.
\end{enumerate}
\end{restatable}

In order to explain our first application of this notion of brambles, we need to recall a result of Seymour~\cite{Seymour2016}. 
He proved that for every integer $k\ge 4$, if a graph $G$ excludes all induced cycles of length at least $k$, then $G$ admits a tree-decomposition in which the subgraph induced by each bag has a dominating path with at most $k-3$ vertices.
If in addition $G$ is $K_{1,d}$-free for some $d\ge 2$, the above property implies that the tree-independence number of $G$ is at most $(d-1)(k-3)$.
Using \Cref{thm:brambles}, this implies the following.

\begin{corollary}\label{thm:from-Seymour}
Let $d\ge 2$ and $k\ge 4$ be integers and $G$ be a $K_{1,d}$-free graph with a strong bramble $\mathcal{B}$ of $\alpha$-order at least $dk$.
Then there exists an induced cycle $C$ of length at least $k$ in $G$.
\end{corollary}

We strengthen \Cref{thm:from-Seymour} to tighten the relation between the induced cycle and the bramble as follows.

\begin{restatable}{theorem}{thmlongdominatingholes}\label{cor:longdominatingholes}
    Let $d\geq 1$ and $k\geq 2$ be integers and $G$ be a $K_{1,d}$-free graph with a strong bramble $\mathcal{B}$ of $\alpha$-order at least $dk$.
    Then there exists an induced cycle $C$ of length at least $k$ in $G$ such that $N[C]\cap B\neq\emptyset$ for all $B\in \mathcal{B}$.
\end{restatable}

We believe that understanding the structural properties of strong brambles of large $\alpha$-order is key for a possible resolution of the conjecture of Dallard et al.
In our case, the long induced cycle we are guaranteed to find in relation to any strong bramble of large $\alpha$-order due to \Cref{cor:longdominatingholes} directly allows us to construct a large induced wheel minor in any $K_{1,d}$-free graph of large tree-independence number.

\subsection{Related work}

The \textsl{Grid Theorem} of Robertson and Seymour \cite{GraphMinor5} can be understood as a min-max relation between the treewidth of a graph and its largest grid minor.
Recently, considerable attempts have been made to understand treewidth through a similar min-max relation using induced subgraphs (see for example \cite{alecu2023induced}).
Replacing treewidth with tree-independence number often allows for more streamlined arguments for the study of forbidden induced objects \cite{dallard2024treewidthversuscliquenumber2,chudnovsky2024tree}.
We believe that the existence of canonical obstructions such as the brambles discussed in this paper will be crucial in further streamlining arguments about tree-independence number and forbidden induced substructures.

For graphs $G$ of bounded degree, i.e.\@ a proper subclass of graphs of bounded clique number, by the arguments above there is only a qualitative difference between\footnote{We denote the treewidth of a graph $G$ by $\mathsf{tw}(G)$.} $\mathsf{tw}(G)$ and $\alpha\text{-}\mathsf{tw}(G)$.
Hence, the study of both parameters essentially coincides within classes of bounded degree.
Korhonen \cite{Korhonen2023InducedGrid} proved that for a graph with bounded maximum degree, if its treewidth is sufficiently large, the graph must contain a large grid as an induced minor.

\begin{theorem}[Korhonen \cite{Korhonen2023InducedGrid}]\label{prop:inducedgrid}
There exists a function $f_{\ref{prop:inducedgrid}}(d,k)\in\mathcal{O}(k^{10}+2^{d^5})$ such that every graph $G$ with maximum degree at most $d$ and $\mathsf{tw}(G)\geq f_{\ref{prop:inducedgrid}}(d,k)$ contains the $(k\times k)$-grid as an induced minor.
\end{theorem}

Notice that graphs of maximum degree at most $d-1$ are exactly those that do not contain $K_{1,d}$ as a \textsl{subgraph}.
A natural generalization of this property to the world of induced subgraphs would be to exclude $K_{1,d}$ as an \textsl{induced} subgraph.
In other words, for each vertex in a $K_{1,d}$-free graph, its neighbourhood has bounded independence number.
Thus, we can consider $K_{1,d}$-free graphs as those of bounded ``induced'' maximum degree.
It is apparent that the equivalence of treewidth and tree-independence number can no longer hold: the class of complete graphs is contained in the class of $K_{1,2}$-free graphs, has unbounded treewidth but tree-independence number bounded by $1$.
Building on this intuition, in \cite{dallard2024treewidthversuscliquenumber4}, Dallard et al.\@ conjectured an analogue of \Cref{prop:inducedgrid} for $K_{1,d}$-free graphs.

\begin{conjecture}[Dallard et al.~\cite{dallard2024treewidthversuscliquenumber4}]\label{conj:K1dfreeinducedgrid}
    There exists a function $f(d,k)$ such that every $K_{1,d}$-free graph $G$ with $\alpha$-$\mathsf{tw}(G)\geq f(d,k)$ contains the $(k\times k)$-grid as an induced minor.
\end{conjecture}

In the same paper, Dallard et al.\@ provide partial evidence towards their conjecture.
They prove two results, the first is that every $\{K_{1,d},P_k\}$-free graph $G$ (with $d\ge 2$ and $k\ge 3$) satisfies $\alpha$-$\mathsf{tw}(G)\leq(d-1)(k-2)$ and the second is that every $K_{1,d}$-free graph that excludes a $k$-times subdivided claw\footnote{For $k\in\mathbb{N}$, a \emph{$k$-times subdivided claw} is a graph obtained from $K_{1,3}$ by subdividing every edge $k$ times.} as an induced minor has tree-independence number bounded by some function depending only on $d$ and $k$ (see also~\cite{chudnovsky2025treeindependencenumberv} for a more general result).
Recently, Chudnovsky, Hajebi, and Trotignon \cite{chudnovsky2024tree} proved that excluding so-called \textsl{thetas} and \textsl{prisms} from $K_{1,d}$-free graphs also results in classes of bounded tree-independence number.
Furthermore, Ahn, Gollin, Huynh, and Kwon \cite{ahn2024coarseerdhosposatheorem} proved that there is a function $f$ such that for every $k\in\mathbb{N}$, every $K_{1,d}$-free graph $G$ either contains $k$ pairwise disjoint and pairwise non-adjacent induced cycles, or there is a set $X\subseteq V(G)$ such that $\alpha(G[X])\leq f(k)$ and $G-X$ is chordal, implying that $\alpha\text{-}\mathsf{tw}(G)\leq f(k)+1$.

The results above are the first that provide evidence towards \Cref{conj:K1dfreeinducedgrid} by showing that the exclusion of small planar patterns forces bounded tree-independence number in $K_{1,d}$-free graphs.
The wheels $W_{\ell}$, $\ell\geq 3$, while containing long induced paths and cycles, are in general incomparable with respect the induced minor relation with both the thetas and prisms of Chudnovsky et al.\@ and the subdivided claws of Dallard et al.
In this regard, our main result, \Cref{thm:forbiddingthewheel}, presents a novel structure whose exclusion forces small tree-independence number in $K_{1,d}$-free graphs.
Moreover, wheels mark the first infinite class of \textsl{$3$-connected} planar graphs whose exclusion as induced minors bounds the tree-independence number of $K_{1,d}$-free graphs.
With grids being (almost) $3$-connected, this can be considered to be a major milestone towards a (hopefully positive) resolution of \Cref{conj:K1dfreeinducedgrid}.

\section{Preliminaries}\label{sec:Prelim}

All logarithms in this paper are base $2$.
Let $a,b\in\mathbb{Z}$.
We denote the set $\{ a,a+1,\dots,b-1,b\}$ by $[a,b]$.
Notice that $[a,b]=\emptyset$ in case $a>b$.
Moreover, we further abbreviate the set $[1,a]$ by $[a]$.

All graphs we consider in the paper are finite, simple, and undirected.
For a graph $G$, let $V(G)$ denote the vertex set of $G$, and let $E(G)$ denote the edge set of $G$.

Let $G$ and $H$ be graphs.
We say that $H$ is an \emph{induced subgraph} of $G$ if $H$ can be obtained from $G$ by a sequence of vertex deletions.
If $H$ is a subgraph of $G$ we write $G-H$ for $G-V(H)$.

Let $G$ and $H$ be graphs and $e=xy\in E(G)$ be an edge.
We say that $H$ is obtained from $G$ by \emph{contracting} $e$ if there exists a vertex $u\in V(H)$ such that $H-u = G-x-y$ and $N_H(u)= N_G(x)\cup N_G(y)$.
We say that $H$ is an \emph{induced minor} of $G$ if $H$ can be obtained from $G$ by a sequence of vertex deletions and edge contractions.

Given a set $\mathcal{H}$ of graphs, we say that a graph $G$ is \emph{$\mathcal{H}$-free} if there does not exist an $H\in\mathcal{H}$ such that $G$ contains $H$ as an induced subgraph.
In case $\mathcal{H}=\{ H\}$ we also write \emph{$H$-free} instead of $\mathcal{H}$-free.

Let $G$ be a graph and $X\subseteq V(G)$.
We denote by $N_G(X)$ the set $\{ v\in V(G)\setminus X \colon vx\in E(G)\text{ for some }x\in X \}$.
We call $N_G(X)$ the \emph{neighbourhood} of $X$ in $G$.
Moreover, $N_G[X]\coloneqq N_G(X)\cup X$ is called the \emph{closed neighbourhood} of $X$.
By a slight abuse of notation, if $H$ is an induced subgraph of $G$, we write $N_G[H]$ (respectively $N_G(H)$) for $N_G[V(H)]$ (respectively $N_G(V(H))$).
In case $G$ is understood from the context, we write $N$ instead of $N_G$ and, for $X\subseteq V(G)$, we write $\alpha(X)$ to denote the value $\alpha(G[X])$ for the maximum size of an independent set in the subgraph of $G$ induced by $X$.

\paragraph{Trees, roots, and tree-decompositions.}
A \emph{rooted forest} is a pair $(F,R)$ where $F$ is a forest $R$ is a set of vertices, one from each connected component of $F$.
The vertices in $R$ are called the \emph{roots} of $F$.
If $F$ has a unique component and $R=\{r\}$ we may write $(F,r)$ instead of $(F,R)$, and in this case $(F,r)$ is a \emph{rooted tree}.
When specifying a rooted forest, we may omit $R$ and write only $F$ if the context permits it.
Let $(F,R)$ be a rooted forest. 
An \emph{ancestor} of a vertex $v\in V(F)$ is a vertex distinct from $v$ that belongs to the unique path from $v$ to a root. 
Furthermore, a \emph{descendant} of a vertex $v\in V(F)$ is a vertex that admits $v$ as an ancestor.
In a (rooted) forest, a vertex with at most one neighbour is called a \emph{leaf}, and a vertex with at least two neighbours, or that is a root in the rooted case, is called an \emph{internal vertex}.

A \emph{tree-decomposition} for a graph $G$ is a tuple $\mathcal{T}=(T,\beta)$ where $T$ is a tree and $\beta\colon V(T)\to 2^{V(G)}$, called the \emph{bags} of $\mathcal{T}$ (here $2^{V(G)}$ denotes the power set of the vertex set of $G$), such that
\begin{enumerate}
	\item $\bigcup_{t\in V(T)}\beta(t)=V(G)$,
	\item for every $e\in E(G)$ there exists some $t\in V(T)$ such that $e\subseteq \beta(t)$, and
	\item for every vertex $v\in V(G)$ the set $\{ t\in V(T) \colon v\in\beta(T) \}$ induces a subtree of $T$.
\end{enumerate}
The \emph{width} of $\mathcal{T}$ is defined as the value $\max_{t\in V(T)}|\beta(t)|-1$ and the \emph{treewidth} of $G$, denoted by $\mathsf{tw}(G)$, is the minimum width over all tree-decompositions for $G$.
The \emph{$\alpha$-width} of $\mathcal{T}$ is defined as the value $\alpha$-$\mathsf{width}(\mathcal{T})=\max_{t\in V(T)}\alpha(\beta(t))$ and the \emph{tree-independence number} of $G$, denoted by $\alpha$-$\mathsf{tw}(G)$, is the minimum $\alpha$-width over all tree-decompositions for $G$.

\section{A qualitative duality for \texorpdfstring{$\alpha$-treedepth}{α-treedepth} in \texorpdfstring{$K_{1,d}$}{K\_\{1,d\}}-free graphs}\label{sec:treedepth}

As a warm-up, we strengthen a result from \cite{dallard2024treewidthversuscliquenumber4} stating that $K_{1,d}$-free graphs without long induced paths have bounded tree-independence number.
We show that in this statement, tree-independence number can be replaced by a new parameter, $\alpha$-treedepth, which acts as an ``independence'' counterpart to the classic notion of treedepth.

\subsection{The definition of \texorpdfstring{$\alpha$-treedepth}{α-treedepth} and preliminary results}

An \emph{elimination forest} for a graph $G$ is a rooted forest $F$ such that $V(F)=V(G)$ and for every edge $uv\in E(G)$ there exists a root-to-leaf path $P$ in $F$ such that $u,v\in V(P)$.
The \emph{$\alpha$-depth} of a rooted forest $F$ is the maximum value of $\alpha(G[V(P)])$ over all root-to-leaf paths $P$ in $F$. 
The \emph{$\alpha$-treedepth} of a graph $G$, denoted by $\atd$ is the minimum $\alpha$-depth over all elimination forests for $G$.

We start by characterising the graphs with $\alpha$-treedepth at most $1$.
By definition, a graph $G$ has $\alpha$-treedepth at most $1$ if and only if $G$ admits an elimination forest $F$ such every root-to-leaf path $P$ in $F$ induces a clique in $G$. 
In this case, we say that $G$ is the \emph{transitive closure} of $F$.
In \cite[Theorem 3]{Yan1996Quasithreshold}, the authors give several equivalent characterisations of the \emph{quasi-threshold} graphs (also known as \emph{trivially perfect graphs}). 
Among other interesting characterisations, this class of graphs is exactly the class of the $\{P_4,C_4\}$-free graphs.
(We denote by $C_k$ the $k$-vertex cycle.)
Furthermore, $G$ is a quasi-threshold graph if and only if $G$ is the transitive closure of some rooted forest. 
Thus we can state this characterisation in terms of $\alpha$-treedepth.

\begin{theorem}
For every graph $G$, $G$ has $\atd(G) \le 1$ if and only if $G$ is $\{P_4,C_4\}$-free graph.
\end{theorem}

Next, we establish some basic properties of $\alpha$-treedepth.

\begin{observation}\label{obs:treedepth1}
If $H$ is an induced subgraph of $G$, then $\atd(H)\leq \atd(G)$.
\end{observation}

\begin{observation}\label{obs:treedepth2}
Let $G$ be a graph obtained from the disjoint union of $G_1$ and $G_2$. Then $\atd(G)=\max\{\atd(G_1),\atd(G_2)\}$.
\end{observation}

\begin{lemma}\label{lemma:twleqtd}
For every graph $G$ it holds that $\alpha\text{-}\mathsf{tw}(G)\leq \alpha\text{-}\mathsf{td}(G)$.
\end{lemma}

\begin{proof}
It suffices to prove the claim for connected graphs, so we may assume $G$ to be connected.
Let $F$ be an elimination forest of minimum $\alpha$-depth for $G$.
Since $G$ is connected, $F$ is a tree.
Let $L$ be the set of all leaves of $F$ and let $\lambda$ be a linear ordering of the leaves obtained by the order of encountering them when following a depth-first search from the root of $F$.
Now let $P$ be a path with $V(P)=L$ such that $\lambda$ is exactly the order of the vertices of $L$ encountered when traversing along $P$ from one endpoint to the other.
For each $p\in V(P)$ we define $\beta(p)$ to be the set of all vertices contained in the unique root-to-$p$ path in $F$.
It follows that $\alpha(\beta(p))\leq \alpha\text{-}\mathsf{depth}(F)$.
So all that is left is to show that $(P,\beta)$ is a tree-decomposition for $G$.
The first two conditions from the definition of tree-decompositions hold trivially by the definition of $(P,\beta)$ and the fact that $F$ is an elimination tree for $G$.
Now suppose there exists a vertex $v\in V(G)$ and leaves $l_1,l_2,l_3$ of $F$ such that $\lambda(l_1)\leq \lambda(l_2)\leq \lambda(l_3)$ but $v\in\beta(l_1)\cap\beta(l_3)$ and $v\notin \beta(l_2)$.
Then $v$ is a common ancestor of $l_1$ and $l_3$ but not an ancestor of $l_2$.
Since we chose $\lambda$ to originate from a depth-first search, $v$ is an ancestor of $l_1$ but not of $l_2$, and $\lambda(l_1)\leq \lambda(l_2)$, all descendants of $v$ among the vertices in $L$ must come before $l_2$ in $\lambda$, which renders the assumption that $\lambda(l_2)\leq\lambda(l_3)$ absurd.
\end{proof}

\subsection{Main result on \texorpdfstring{$\alpha$-treedepth}{α-treedepth}}

The main result of this section is an asymptotic characterisation of $\alpha$-treedepth for $K_{1,d}$-free graphs.
The main technique for our proof is a type of argument that has become known in the community as the ``Gy\'arf\'as path argument''.
An argument reminiscent of this technique will later appear in \cref{sec:pathsandcyclesinbrambles} where we show that brambles of large $\alpha$-order imply the existence of induced cycles.
These types of arguments are very constructive.
In our case we obtain a greedy procedure to approximate $\alpha$-treedepth in $K_{1,d}$-free graphs.

\thmalphatreedepth*

Notice that \textsl{i)} implies, by \cref{obs:treedepth1}, that every graph that contains an induced $P_k$ has $\alpha$-treedepth at least $\log(k/3)$.
In this sense, \textsl{i)} acts as a lower bound on the $\alpha$-treedepth of graphs.
In the same sense, \textsl{ii)} acts as an upper bound on the $\alpha$-treedepth on $K_{1,d}$-free graphs without long induced paths.
Together these two statements provide an approximate duality theorem for the $\alpha$-treedepth of $K_{1,d}$-free graphs.

\begin{proof}[Proof of \Cref{thm:alphatreedepth}]
We prove the two assertions of \cref{thm:alphatreedepth} independently.
\smallskip

We begin with the lower bound, that is, we show that graphs with long induced paths have large $\alpha$-treedepth.

\textbf{Statement i)}

We will first prove that $\atd(P_{k})\leq \ell$ where $k=3\cdot (2^\ell-1)$ for some $\ell\in \mathbb{N}$.
We inductively construct $T_{k}$ as follows (see \cref{fig:atdofpath}).

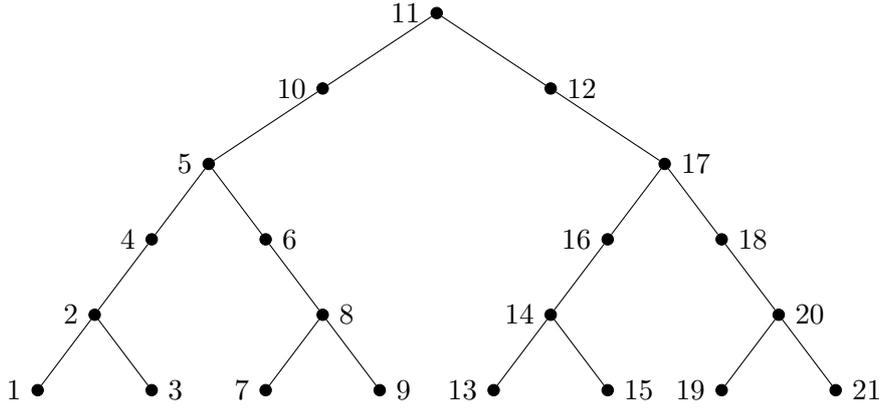
\begin{figure}
    \centering
    \begin{tikzpicture}
      \tikzset{enclosed/.style={draw, circle, inner sep=0pt, minimum size=.15cm, fill=black}}
      \node[enclosed] (11) at (0,0) [label=left: 11] {};
      \node[enclosed] (10) at (-1.5,-1) [label=left: 10] {};
      \node[enclosed] (12) at (1.5,-1) [label=right: 12] {};
      \node[enclosed] (5) at (-3,-2) [label=left: 5] {};
      \node[enclosed] (17) at (3,-2) [label=right: 17] {};
      \node[enclosed] (4) at (-3.75,-3) [label=left: 4] {};
      \node[enclosed] (6) at (-2.25,-3) [label=right: 6] {};
      \node[enclosed] (16) at (2.25,-3) [label=left: 16] {};
      \node[enclosed] (18) at (3.75,-3) [label=right: 18] {};
      \node[enclosed] (2) at (-4.5,-4) [label=left: 2] {};
      \node[enclosed] (8) at (-1.5,-4) [label=right: 8] {};
      \node[enclosed] (14) at (1.5,-4) [label=left: 14] {};
      \node[enclosed] (20) at (4.5,-4) [label=right: 20] {};
      \node[enclosed] (1) at (-5.25,-5) [label=left: 1] {};
      \node[enclosed] (3) at (-3.75,-5) [label=right: 3] {};
      \node[enclosed] (7) at (-2.25,-5) [label=left: 7] {};
      \node[enclosed] (9) at (-0.75,-5) [label=right: 9] {};
      \node[enclosed] (13) at (0.75,-5) [label=left: 13] {};
      \node[enclosed] (15) at (2.25,-5) [label=right: 15] {};
      \node[enclosed] (19) at (3.75,-5) [label=left: 19] {};
      \node[enclosed] (21) at (5.25,-5) [label=right: 21] {};
      \draw (11) -- (10) node[] () {};
      \draw (11) -- (12) node[] () {};
      \draw (10) -- (5) node[] () {};
      \draw (12) -- (17) node[] () {};
      \draw (5) -- (4) node[] () {};
      \draw (5) -- (6) node[] () {};
      \draw (17) -- (16) node[] () {};
      \draw (17) -- (18) node[] () {};
      \draw (4) -- (2) node[] () {};
      \draw (6) -- (8) node[] () {};
      \draw (16) -- (14) node[] () {};
      \draw (18) -- (20) node[] () {};
      \draw (2) -- (1) node[] () {};
      \draw (2) -- (3) node[] () {};
      \draw (8) -- (7) node[] () {};
      \draw (8) -- (9) node[] () {};
      \draw (14) -- (13) node[] () {};
      \draw (14) -- (15) node[] () {};
      \draw (20) -- (19) node[] () {};
      \draw (20) -- (21) node[] () {};
    \end{tikzpicture}
    \caption{Optimal elimination tree for $\atd(P_{21})=3$}
    \label{fig:atdofpath}
\end{figure}

We identify $V(P_{k})$ with $[k]$, with two vertices $i,j\in [k]$ adjacent if and only if $|i-j| = 1$.
If $\ell=1$ (and $k=3$), $T_\ell$ is a $P_3$ with $2$ as a root.
It is easy to check that this is indeed an elimination tree of $P_3$, which implies that $\alpha\text{-}\mathsf{depth}(P_3)=1$.

Suppose now that $\ell\geq 2$.
Let $r\coloneqq(k+1)/2$ be the root vertex of $T_{\ell}$, and let $r-1$ and $r+1$ be its children.
Add two copies of $T_{\ell-1}$ by joining the root vertex of $T_{\ell-1}$ as a child to each of $r-1$ and $r+1$, adding $(k+3)/2$ to all labels of the copy of $T_{\ell-1}$ attached to $r+1$.
Then $T_\ell$ is an elimination tree of $P_{k}$.
Furthermore, since the vertex set of each root-leaf path of $T_\ell$ is the union of $\ell$ pairs of two consecutive vertices, the independence number of the corresponding induced subgraph of $P_{k}$ is at most $\ell$.
Hence, we have $\atd(P_{k})\leq \ell$.

For general $k$, \cref{obs:treedepth1} gives that $\atd(P_k)\leq \lceil \log\big(\frac{k}{3}+1\big) \rceil$ for all $k$.

\medskip
We now prove that $\atd(P_k)\geq \lceil \log\big(\frac{k}{3}+1\big) \rceil$ by induction on $k$.
For the base cases, if $k\in[3]$, it is easy to see that $\atd(P_k)=1$.
Hence, we may assume that $k\geq 4$.

Let us fix some notations on rooted forests.
For a rooted tree $(T,r)$ and $S\subseteq V(T)$, define $T(S)$ to be the rooted forest with vertex set $S$, in which two distinct vertices $x$ and $y$ are adjacent if and only if one of them is an ancestor of the other in $T$ and the unique $x,y$-path in $T$ contains no other vertex from $S$.
The set of roots of $T(S)$ is the set $R(S)\coloneqq \{ r\in S \colon \text{no ancestor of }r\text{ belongs to }S \}$.
Note that if $T$ is an elimination forest of a graph $G$, then $(T(S),R(S))$ is an elimination forest of $G[S]$.
For $s\in V(T)$, let $D(s)$ be the set consisting of $s$ and all its descendants in $(T,r)$.

Let $(T,r)$ be an optimal elimination tree for the path $P_k$ with vertex set $[k]$ witnessing $\atd(P_k)$.
Let $A=[r-1]$ and $B=[r+1,k]$.

First we claim that we may assume $r$ has exactly two children, and their sets of descendants are $A$ and $B$, respectively.
Suppose that $r$ has a child $s$ such that $D(s)\cap A\neq \emptyset$ and $D(s)\cap B\neq \emptyset$.
Then we can construct a new elimination tree $(T',r)$, by connecting $r$ to the roots of the elimination forests $T(D(s)\cap A)$ and $T(D(s)\cap B)$.
The vertex set of any root-to-leaf path of $(T',r)$ is a subset of a root-to-leaf path of $(T,r)$, so the $\alpha$-depth of $(T',r)$ is at most the $\alpha$-depth of $(T,r)$.
Hence, we may consider $(T',r)$ instead and therefore we may assume that for each child $s$, $D(s)\subseteq A$ or $D(s)\subseteq B$.

Next, suppose that $r$ has more than two children, say $v_1,\ldots,v_{\ell}$, where $\ell\geq 3$.
Without loss of generality, we may assume that $r-1\in D(v_1)$ and $D(v_1)\subsetneq A$ and, similarly, $r+1\in D(v_2)$ and $D(v_2)\subseteq B$.
Since $D(v_1)\subsetneq A$, there exists a child $v_j$ of $r$ with $j\ge 3$ such that $D(v_j)\subseteq A$; without loss of generality we may assume that $j =3$.
Let $s\coloneqq\max D(v_3)$.
Since $r-1\in D(v_1)$ we know that $s\leq r-2$ and thus in $P_k$, the vertex $r$ cannot be a neighbour of $s$ and $s$ has no neighbour in $B$.
It follows that $s+1\in A$ and thus, there is some $i\in[\ell]\setminus\{ 3\}$ such that $s+1\in D(v_i)\subseteq A$.
However, no root-to-leaf path contains the edge between $s$ and $s+1$, which contradicts the definition of elimination forests.
Hence, we may assume that $r$ has two children, and the sets corresponding to their descendant plus themselves are $A$ and $B$, respectively.

Next, we claim that we may assume $r-1$ and $r+1$ to be the two children of $r$.
Towards a contradiction, suppose this is not the case.
We construct a new elimination tree $(T',r)$ as follows: 
First, let $r-1$ and $r+1$ be the two children of $r$.
Attach the roots of $T(A\setminus \{r-1\})$ to $r-1$ and the roots of $T(B\setminus \{r+1\})$ to $r+1$ as descendants.
Consider a root-to-leaf path $P'$ of $T'$ containing $r-1$.
Its vertex set coincides with the vertex set of some root-to-leaf path $P$ in $T$, except that it may additionally contain $r-1$.
However, this does not change the independence number, since for any independent set $I$ in the subgraph of $P_k$ induced by $V(P')$ such that $r-1\in I$, replacing $r-1$ with $r$ in $I$ results in an equally large independent set in the subgraph of $P_k$ induced by $V(P)$.
This works in same way for the root-to-leaf paths containing $r+1$.
Therefore, the $\alpha$-depth does not increase under this operation.

The above assumptions on the structure of $(T,r)$ imply that $\atd(P_{r-2})\le \atd(P_k)-1$, and, similarly, that $\atd(P_{k-r-1})\le \atd(P_k)-1$.
Hence,
\begin{align*}
\atd(P_k)&\geq \max\{\atd(P_{r-2}),\atd(P_{k-r-1})\}+1\\
&= \atd(P_{\max\{r-2,k-r-1\}})+1\\
&\geq \atd\left(P_{\lceil(k-3)/2\rceil}\right)+1\,.
\end{align*}
We can now use induction on $k$ to deduce that $\atd(P_k)\ge \lceil \log\big(\frac{k}{3}+1\big) \rceil$.
\bigskip

\textbf{Statement ii)}

Without loss of generality we may assume $G$ to be connected.

Observe that if $d\leq1$ or $k\leq 2$, if $G$ is $\{P_k,K_{1,d} \}$-free, then $G$ has no edge, thus the bound $\atd(G)\leq 1$ holds trivially.
So we may assume that $d\geq 2$ and $k\geq 3$, thus $(d-1)(k-2)\geq 1$.

We construct an elimination tree $(T,r)$ of $G$ together with an induced path with the following induction on $k\geq 3$.

\textbf{Induction Hypothesis.}
For every integer $d\geq 2$ and $k\geq 3$, every connected $K_{1,d}$-free graph $G$, and every vertex $r\in V(G)$, there exists either an elimination tree $(T,r)$ of $G$ of $\alpha$-depth at most $(d-1)(k-2)$ or an induced path in $G$ on $k$ vertices with endpoint $r$.

\smallskip
Let us fix $d\geq 2$, and let $G$ be a connected $K_{1,d}$-free graph, and $r$ an arbitrary vertex of $G$.

Let us first show that the statement holds for $k=3$.
We want to exhibit either an induced $P_3$ with endpoint $r$ or an elimination tree $(T,r)$ of $G$ with $\alpha$-depth at most $d-1$.
Suppose that there is no induced $P_3$ with endpoint $r$.
Since $G$ is connected, this implies that $r$ is adjacent to every vertex of $G$, i.e. $N_G[r]=V(G)$.
Since $G$ is $K_{1,d}$-free, $\alpha (N_G[r])\leq d-1$, so $\alpha (G)\leq d-1$.
Therefore any elimination tree of $G$ has $\alpha$-depth at most $d-1$.

Assume now that $k\geq 4$.
Let $Q$ be a graph isomorphic to a path with vertex set $N_G[r]$ such that $r$ is an endpoint of $Q$ (the edge set of $Q$ is chosen arbitrarily to form that path, we will later use this path in the eventual construction of an elimination tree of $G$).
Let $J_1,\dots,J_{\ell}$ be the components of $G-N_G[r]$.
For every $i\in[\ell]$ we choose $r_i\in N_G(J_i)\subseteq N_G(r)$ such that $r_i$ is the furthest vertex from $r$ along $Q$, and set $G_i\coloneqq G[V(J_i)\cup\{r_i\}]$.

We apply the induction hypothesis for $k-1$ to $G_i$ for each $i\in[\ell]$.
If there exists $j\in[\ell]$ such that $r_j$ is an endpoint of an induced $P_{k-1}$, say $P$, in $G_j$, we have found an induced $P_k$ in $G$ since the only neighbour of $r$ in $P$ is exactly the endpoint $r_j$.
Hence, we may assume that for every $i\in[\ell]$ we find an elimination tree $(T_i,r_i)$ of $G_i$ with $\alpha$-depth at most $(d-1)(k-3)$.

We construct an elimination tree $(T,r)$ of $G$ as follow.
Recall that we constructed $Q$ with $V(Q)=N_G[r]$ and $r$ as an endpoint, so for each $i$, the vertex $r_i$ belongs to $Q$, and it is possible that there exists a $j\neq i$ such that $r_i$ and $r_j$ correspond to the same vertex in $Q$.
Let $T$ be the tree rooted in $r$ obtained from the union of $Q$ and $T_i$ for each $i\in[\ell]$. 

Then $(T,r)$ is an elimination tree for $G$ since $Q$ is a path, each $T_i$ is a subtree of $T$ rooted in $r_i$ and corresponds to an elimination tree of $G_i$, and the path $Q$ allows all the possible edges in $N_G[r]$.
Moreover, since we chose $r_i$ as far as possible from $r$ along $Q$, the set $N_G(J_i)$ corresponds to $r_i$ and its ancestors in $(T,r)$, allowing all the possible edges between $J_i$ and $N_G(J_i)$.
Finally, let us show that the $\alpha$-depth of $(T,r)$ is at most $(d-1)(k-2)$.
To see this, let $R$ be an arbitrary root-to-leaf path in $T$.
Then, by definition, there exists $i\in[\ell]$ such that $R-Q$ is a subpath of a root-to-leaf path $R'$ in $T_i'$.
Since $V(Q)=N_G[r]$ and $\alpha(N_G(r))\leq d-1$ we have that $\alpha(V(R))\leq \alpha(V(Q)) + \alpha(V(R'))\leq d-1 + (d-1)(k-3) = (d-1)(k-2)$.

\end{proof}

\section{Bramble duality for tree-independence number}\label{sec:duality}

Our goal is to prove an approximate duality between tree-independence number and a suitable notion of brambles that allows us to easily show lower bounds and provide us with additional information on the structure of graphs with large tree-independence number.
The notion of brambles we introduce here is slightly more restrictive than the usual definition (see, e.g., \cite{diestel2016graph}), hence we add the attribute \textsl{strong}.
We make this choice here because it allows slightly simpler arguments while additionally avoiding Menger's Theorem (which is used to prove the duality between brambles and treewidth, but cannot be adapted for independence number).
For a proof of bramble/treewidth duality see, e.g., \cite[Theorem 12.4.3]{diestel2016graph}.

Let $G$ be a graph.
We say that a set $X\subseteq V(G)$ is \emph{connected} if $G[X]$ is connected.

A \emph{strong bramble} in $G$ is a collection $\mathcal{B}=\{ B_1,\dots,B_n\}$ of connected sets $B_i\subseteq V(G)$ such that $B_i\cap B_j\neq\emptyset$ for all $i,j\in[n]$.

A \emph{cover} for a strong bramble $\mathcal{B}=\{ B_1,\dots,B_n\}$ is a set $X\subseteq V(G)$ such that $B_i\cap X\neq\emptyset$ for all $i\in[n]$.
The \emph{$\alpha$-order} of $\mathcal{B}$ is defined as $\min\{ \alpha(X) ~\!\colon\!~  X\subseteq V(G)\text{ is a cover for }\mathcal{B}\}$.

Our proof for \cref{thm:brambles} takes a slight detour through the concept of balanced separators and highly linked sets, i.\@e.\@ those kinds of sets that do not have such separators.
Our proofs for the results in this section follow along very well established strategies (see, e.g.\@ \cite{IsoldeThesis}).

Let $G$ be a graph and $X\subseteq V(G)$ be a set of vertices.
We say that a set $S\subseteq V(G)$ is an \emph{$\alpha$-balanced separator} for $X$ if every component $C$ of $G-S$ satisfies
$$\alpha(V(C)\cap X) \leq \frac{\alpha(X)}{2}.$$

Now let $k\geq 1$ be an integer and $G$ be a graph.
A set $X\subseteq V(G)$ is \emph{$k$-$\alpha$-linked} if there does not exist an $\alpha$-balanced separator $S$ for $X$ with $\alpha(S)\leq k$.

\begin{observation}\label{obs:uniquecomponent}
Let $k\geq 1$ be an integer and $G$ be a graph with a $k$-$\alpha$-linked set $X$.
If $S\subseteq V(G)$ with $\alpha(S)\leq k$ then $G-S$ has a unique component $C_S$ such that $\alpha(V(C_S)\cap X)>\frac{\alpha(X)}{2}$.
\end{observation}

\begin{proof}
Since $X$ is $k$-$\alpha$-linked, there must exist at least one component $C$ of $G-S$ with $\alpha(V(C)\cap X)>\frac{\alpha(X)}{2}$.
Suppose $C$ is not unique and there exists another component $C'$ of $G-S$ satisfying the same inequality.
Then $C$ and $C'$ are disjoint and no vertex of $C$ is adjacent to a vertex of $C'$.
It follows that the union of any independent set in $C$ with any independent set in $C'$ forms an independent set in $G$.
Hence, we deduce that
\begin{align*}
    \alpha\big((V(C)\cup V(C'))\cap X \big) = \alpha(V(C)\cap X) + \alpha(V(C')\cap X) > \alpha(X),
\end{align*}
a contradiction.
\end{proof}

\begin{theorem}\label{thm:welllinkedduality}
Let $k\geq 1$ be an integer.
If a graph $G$ does not contain a $k$-$\alpha$-linked set, then $\alpha$-$\mathsf{tw}(G)\leq 2k+1$.
\end{theorem}

\begin{proof}
A tree-decomposition of \emph{partial $\alpha$-width} at most $2k+1$ for a graph $G$ is a tree-decomposition $(T,\beta)$ of $G$ with a root $r\in V(T)$ such that $\alpha(\beta(t))\leq 2k+1$ for every internal vertex $t\in V(T)$.
The set of \emph{treated vertices} of $(T,\beta)$, denoted by $B(T,\beta,r)$, is the set of all vertices $u$ of $G$ such that there is $t\in V(T)$ with $u\in\beta(t)$ and $\alpha(\beta(t))\leq 2k+1$.

We will now produce a tree-decomposition of $\alpha$-width at most $2k+1$ for $G$ by iteratively refining tree-decompositions of partial $\alpha$-width at most $2k+1$ while increasing the number of treated vertices of these tree-decompositions.
To start the inductive process, simply select a non-empty set $X\subseteq V(G)$ with $\alpha(X)\leq 2k+1$ and set $\beta(r)\coloneqq X$.
Then introduce a vertex $t$ and the edge $rt$ and set $\beta(t)\coloneqq V(G)$.
We thus have $V(T) = \{r,t\}$.
Observe that the resulting $(T,\beta)$ is indeed a tree-decomposition of $G$.
Moreover, the partial $\alpha$-width of $(T,\beta)$ is at most $2k+1$, every vertex of $X=\beta(r)$ is treated.

Now suppose we are given a tree-decomposition $(T,\beta)$ of $G$ of partial $\alpha$-width at most $2k+1$ with root $r$ and $n\geq 1$ treated vertices.
In case the $\alpha$-width of $(T,\beta)$ is at most $2k+1$, we are done.
Hence, we may assume that there exists a leaf $\ell\in V(T)$ such that $\alpha(\beta(\ell))>2k+1$.
Let $e=t\ell$ be the unique edge of $T$ incident with $\ell$.

If $\alpha(\beta(t)\cap\beta(\ell))<2k+1$, select an arbitrary untreated vertex $v\in \beta(\ell)\setminus \beta(t)$ and let $X\coloneqq (\beta(t)\cap\beta(\ell))\cup \{ v\}$.
Notice that by the choice of $\ell$, such a vertex $v$ must exist.
Then introduce a vertex $\ell'$ adjacent to $\ell$ and let $T'$ be the resulting tree.
We define $\beta'(t')\coloneqq \beta(t')$ for all $t'\in V(T)\setminus \{ \ell\}$, $\beta'(\ell)\coloneqq X$, and $\beta'(\ell')\coloneqq \beta(\ell)$.
Then, $(T',\beta')$ is a tree-decomposition of $G$, still rooted at $r$, such that the number of treated vertices of $(T',\beta')$ is strictly larger than the number of treated vertices of $(T,\beta)$.

Otherwise, we have $\alpha(\beta(t)\cap\beta(\ell))=2k+1$ and set $X\coloneqq \beta(t)\cap\beta(\ell)$.
Since $G$ does not contain a $k$-$\alpha$-linked set, $X$ in particular cannot be such a set and thus we can find a set $S\subseteq V(G)$ with $\alpha(S)\leq k$ such that for every component $C$ of $G-S$ we have $\alpha(V(C)\cap X)\leq\frac{\alpha(X)}{2}$.
Let $C_1,\dots,C_q$ be the components of $G-S$ that have a non-empty intersection with $\beta(\ell)$. 
Note that $q\geq 1$ since $\alpha(\beta(\ell)\setminus S)\geq \alpha(\beta(\ell))-\alpha(S)\geq 1$.
For each $i\in[q]$ let $X_i\coloneqq V(C_i)\cap X$ and let $R_i\coloneqq V(C_i)\cap\beta(\ell)$.
Let $S'\coloneq S\cap \beta(\ell)\cap\beta(t)$.
Notice that $X_i$ is a (possibly empty) subset of $R_i$ and $(R_i\cup S', (V(G)\setminus R_i)\cup X_i\cup S')$ is a separation of $G$ where
\[(R_i\cup S')\cap((V(G)\setminus R_i)\cup X_i\cup S')= X_i\cup S'\,.\]
To see this, note that $C_i$ is a component of $G-S$, and by the properties of tree-decompositions, no edge of $C_i-(\beta(\ell)\cap\beta(t))=C_i-X_i$ can have endpoints in both $\beta(\ell)$ and $V(G)\setminus\beta(\ell)$.
For the same reason, no vertex of $R_i\setminus X_i$ can have a neighbour in $S\setminus S'$.
Hence, $R_i\setminus X_i$ is precisely the union of the vertex sets of all components of $C_i-X_i$ with at least one vertex in $\beta(\ell)$.

Moreover, since any independent set in $X_i\cup S'$ decomposes into an independent set in $X_i$ and one in $S'$,
\[\alpha(X_i\cup S')\leq\alpha(X_i)+\alpha(S')\leq \alpha(X_i)+\alpha(S)\leq \left\lfloor \frac{2k+1}{2}\right\rfloor + k= 2k\,.\]
Now, for each $i\in[q]$, pick a vertex $v_i\in R_i\setminus X_i = R_i\setminus (X_i\cup S')$ if possible.
Notice that this implies that $v_i$ is an untreated vertex.
In case there does not exist such a $v_i$, we have $R_i=X_i$ and thus $\alpha(R_i\cup S')\leq 2k$.
For those $i\in[q]$ where this holds, we proceed as for all other cases without choosing $v_i$.
Simply note that in these cases, the resulting new leaves of our decomposition will have small weight with respect to $\alpha$, and in case $R_i=X_i$ for all $i\in[q]$ our procedure will stop as described below.

\begin{figure}
    \centering
    \begin{tikzpicture}[scale=1.5]
    \pgfdeclarelayer{background}
    \pgfdeclarelayer{foreground}
    \pgfsetlayers{background,main,foreground}
    \begin{pgfonlayer}{background}
        \pgftext{\includegraphics[width=10cm]{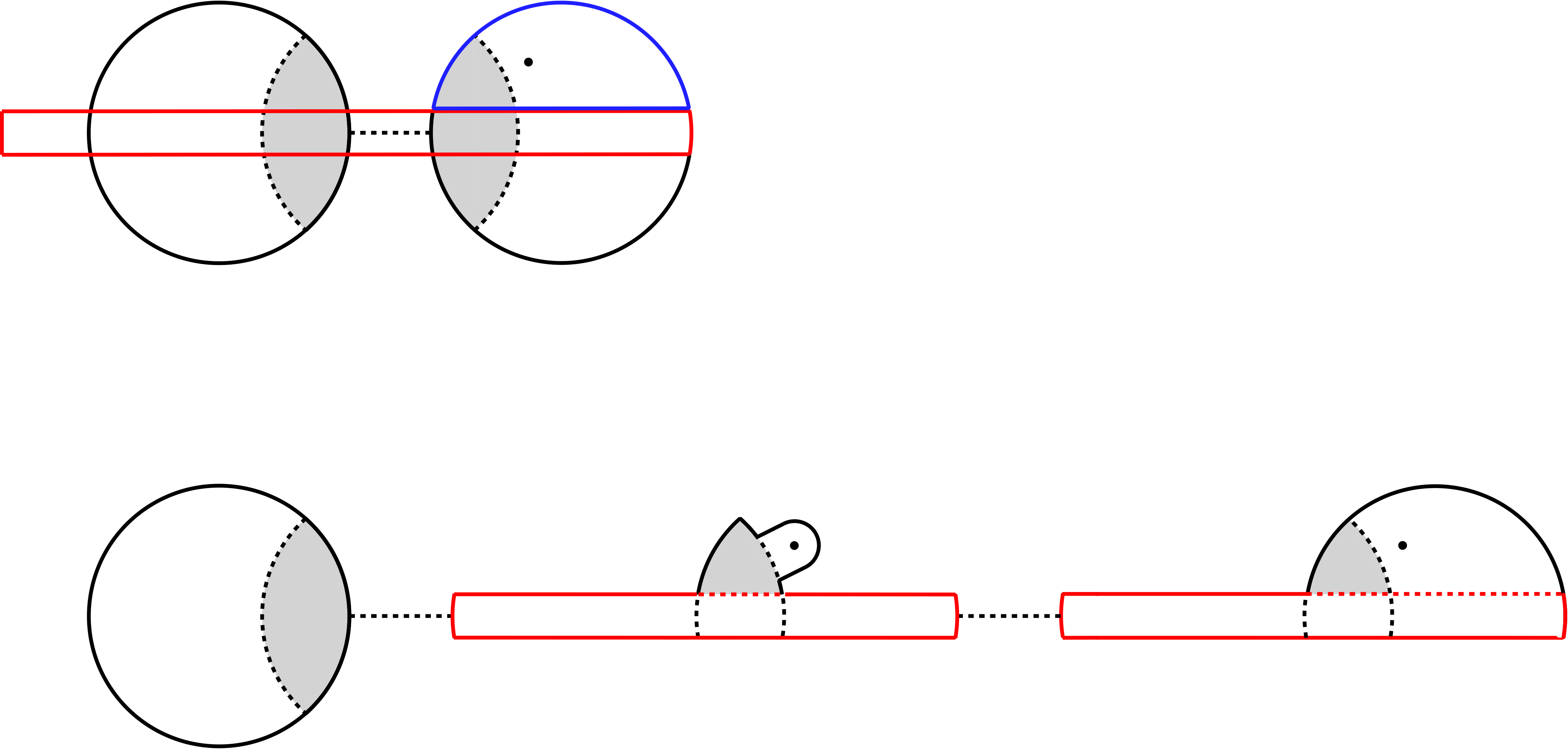}} at (C.center);
    \end{pgfonlayer}{background}
    \begin{pgfonlayer}{main}
        \node[]() at (-1.5,2.1) {$v_i$} ;
        \node[]() at (-5.15,1.55) {\color{red}{$S$}} ;
        \node[]() at (-3.55,2.6) {$\beta(t)$} ;
        \node[]() at (-1.3,2.6) {$\beta(\ell)$} ;
        \node[]() at (-2.75,2.1) {$X$} ;
        \node[]() at (-3.55,-0.5) {$\beta'(t)$} ;
        \node[]() at (-0.5,-0.5) {$\beta'(t_i)=Y_i$} ;
        \node[]() at (-0.86,-1.53) {\color{red}{$S'$}} ;
        \node[]() at (-0.28,-1.25) {$X_i$} ;
        \node[]() at (0.28,-0.92) {$v_i$} ;
        \node[]() at (3.4,-0.5) {$\beta'(\ell_i)$} ;
        \node[]() at (4.07,-1) {$v_i$} ;
        \node[]() at (3.61,-1.25) {$X_i$} ;
        \node[]() at (3,-1.53) {\color{red}{$S'$}} ;
        \node[]() at (-0.5,2.1) {\color{blue}{$R_i$}} ;
        \node[]() at (-2.5,0) {\Huge$\Downarrow$} ;
    \end{pgfonlayer}{main}
    \begin{pgfonlayer}{foreground}
    \end{pgfonlayer}{foreground}
\end{tikzpicture}
    \caption{Construction of a tree-decomposition in \Cref{thm:welllinkedduality}}
    \label{fig:treatedvertices}
\end{figure}
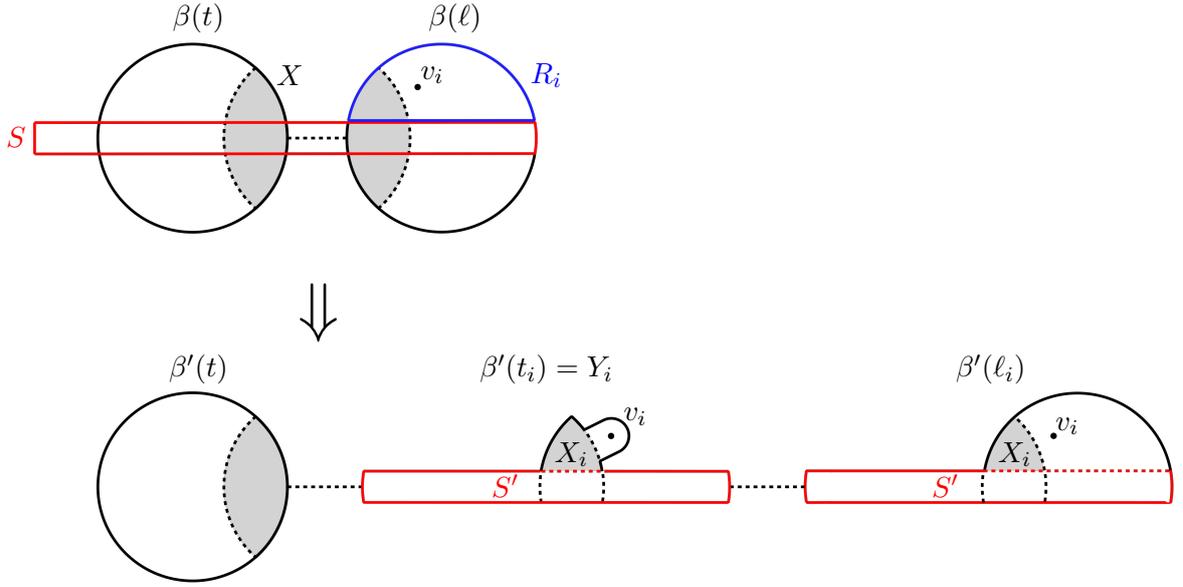

Now, for every $i\in[q]$, let $Y_i\coloneqq \{ v_i\}\cup X_i\cup S'$ (or $Y_i\coloneqq X_i\cup S'$ in case no $v_i$ could be chosen).
Observe that $\alpha(Y_i)\leq 2k+1$.
Finally, for each $i\in[q]$ such that $v_i$ was chosen, introduce a vertex $t_i$ and a vertex $\ell_i$ together with the edges $tt_i$ and $t_i\ell_i$, let $T'''$ be the resulting tree, still with root $r$.
Similarly, for each $i\in[q]$ where $v_i$ could not be chosen, introduce only the vertex $\ell_i$ together with the edge $t\ell_i$ and let $T''$ be the tree obtained from $T'''$ in this way.
In a last step, let $T'\coloneqq T''-\ell$ and set $\beta'(t_i)\coloneqq Y_i$ and $\beta'(\ell_i)\coloneqq R_i\cup Y_i$ where $v_i$ was chosen and $\beta'(\ell_i)\coloneqq Y_i$ otherwise.
For all $t'\in V(T)\setminus \{\ell\}$ set $\beta'(t')\coloneqq \beta(t')$.
See \cref{fig:treatedvertices} for an illustration of this construction.
Then $(T',\beta')$ is a tree-decomposition of $G$ of partial $\alpha$-width at most $2k+1$.
To see this, recall that $(R_i\cup S', (V(G)\setminus R_i)\cup X_i\cup S')$ is a separation of $G$, with separator $X_i\cup S'$, which is equal to $\beta'(t)\cap\beta'(t_i)$ if $v_i$ is chosen, and to $\beta'(t)\cap\beta'(\ell_i)$ otherwise. 
Notice that $\beta'(\ell_i)=R_i\cup S'$ in both cases and $X_i\cup S'\subseteq\beta'(t_i)=Y_i\subseteq \beta'(\ell_i)$ in case $v_i$ was chosen.

Furthermore, the number of treated vertices increases.
Indeed, if there is no $i\in[q]$ such that $v_i$ is chosen, we have $\alpha(\beta'(\ell_i)) = \alpha(X_i\cup S')\leq 2k$ for all $i\in[q]$ and thus $|B(T',\beta',r)|>|B(T,\beta,r)|$ since $\alpha(\beta(\ell))>2k+1$.
Hence, we may assume that there exists some $i\in[q]$ such that $v_i\notin B(T,\beta,r)$ was selected for the bag of the internal vertex $t_i$.
Thus, again we have $|B(T',\beta',r)|>|B(T,\beta,r)|$ and our proof is complete.
\end{proof}


The proof of \Cref{thm:brambles} now consists of two steps.
First, we show that the existence of a strong bramble of large $\alpha$-order forces large tree-independence number.
Second, assume that a given graph $G$ has large tree-independence number, namely, at least $4k-2$.
By \Cref{thm:welllinkedduality} this implies that $G$ contains a $(2k-2)$-$\alpha$-linked set for some $k$.
Finally, notice that for each $S\subseteq V(G)$ with independence number at most $k-1$, $G-S$ has a unique component $J_S$ whose intersection with $S$ has a large independent set.
The collection $\mathcal{B}$ of all of those components $J_S$ can then be observed to be the strong bramble of $\alpha$-order at least $k$, as desired.

\thmbrambles*

We begin with the first step.

\begin{lemma}\label{lemma:bramble_from_linked_set}
Let $k\geq 1$ be an integer.
If a graph $G$ contains a $(2k-2)$-$\alpha$-linked set, then it contains a strong bramble of $\alpha$-order at least $k$.
\end{lemma}

\begin{proof}
Let $X$ be a $(2k-2)$-$\alpha$-linked set in $G$.
Then, for every $S\subseteq V(G)$ with $\alpha(S)\leq 2k-2$ there exists a component $C_S$ of $G-S$ with $\alpha(V(C_S)\cap X)>\frac{\alpha(X)}{2}$.
Notice that $C_S$ is uniquely determined by $S$ due to \cref{obs:uniquecomponent}.
Let $\mathcal{B}\coloneqq\{ V(C_S) \colon S\subseteq V(G),~\alpha(S)\leq k-1 \}$.
We claim that $\mathcal{B}$ is a strong bramble of $\alpha$-order at least $k$.
For this to be true we need to show two things: the members of $\mathcal{B}$ pairwise intersect, and they cannot be covered by a set $Y$ with $\alpha(Y)\leq k-1$.

Let us begin with the pairwise intersection.
Consider any two sets $A,B\subseteq V(G)$ with $V(C_A),V(C_B)\in\mathcal{B}$.
Since $V(C_A),V(C_B)\in\mathcal{B}$ we have $\max\{\alpha(A),\alpha(B)\}\leq k-1$ and thus, since any independent set in $A\cup B$ can be partitioned into an independent set of $A$ and one of $B$, $\alpha(A\cup B)\leq 2k-2$.
Hence, the component $C_{A\cup B}$ of $G-A-B$ as above is defined.
For both $Z\in\{A,B\}$, every component of $G-A-B$ is contained in some component of $G-Z$.
Moreover, by the uniqueness of $C_Z$ it follows that $V(C_{A\cup B})\subseteq V(C_Z)$.
Hence, $\emptyset\neq V(C_{A\cup B})\subseteq V(C_A)\cap V(C_B)$.

To see that no set $A$ with $\alpha(A)\leq k-1$ can intersect all members of $\mathcal{B}$ simply observe that, by definition of $\mathcal{B}$, for every such $A$ there exists a component $C_A$ of $G-A$ with $V(C_A)\in\mathcal{B}$.
It follows that the order of $\mathcal{B}$ must be at least $k$, as desired.
\end{proof}

Now we are ready to proceed with the proof of \cref{thm:brambles}.

\begin{proof}[Proof of \cref{thm:brambles}]
We first prove point \textsl{ii)} of the assertion.
For this suppose $\alpha$-$\mathsf{tw}(G)\geq 4k-2$.
It follows from \cref{thm:welllinkedduality} that there must exist some set $X\subseteq V(G)$ such that $X$ is $(2k-2)$-$\alpha$-linked, as otherwise every $X\subseteq V(G)$ would have an $\alpha$-balanced separator with independence number at most $2k-2$ and so we would have that $\alpha$-$\mathsf{tw}(G)\leq 2(2k-2)+1 = 4k-3 < 4k-2$, which is absurd.
By \cref{lemma:bramble_from_linked_set} this means that there is a strong bramble of $\alpha$-order at least $\frac{2k-2}{2}+1=k$ in $G$.

What is left is to prove \textsl{i)}.
For this, suppose towards a contradiction that $G$ has a strong bramble $\mathcal{B}$ of $\alpha$-order at least $k$ and a tree-decomposition $(T,\beta)$ of $\alpha$-width at most $k-1$.
Let $t_1t_2\in E(T)$ be any edge and for each $i\in[2]$ let $T_i$ be the component of $T-t_1t_2$ that contains $t_i$.
Moreover, let $X_i\coloneqq \bigcup_{t\in V(T_i)}\beta(t)\setminus (\beta(t_1)\cap\beta(t_2))$.
Then, the axioms of tree-decompositions imply that $X_1\cap X_2=\emptyset$.
Since $\alpha(\beta(t_1)\cap\beta(t_2))\leq k-1$, the set $\beta(t_1)\cap\beta(t_2)$ cannot be a cover for $\mathcal{B}$.
So there must be some $B\in\mathcal{B}$ such that $\beta(t_1)\cap \beta(t_2)\cap B=\emptyset.$
As $B$ is connected, it follows from $(T,\beta)$ being a tree-decomposition that there exists a unique $j\in[2]$ such that $B\subseteq X_j$.
Moreover, as $B$ intersects any $B'\in\mathcal{B}$ with $B'\cap\beta(t_1)\cap\beta(t_2)=\emptyset$, every such $B'$ must be contained in $X_j$.

Using this information, we may define an orientation $\vec{T}$ of $T$ such that $(t_1,t_2)\in E(\vec{T})$ if and only if $B\subseteq X_2$ for all $B\in\mathcal{B}$ that are disjoint from $\beta(t_1)\cap\beta(t_2)$.
Now observe that no vertex $t$ of $\vec{T}$ can be incident with two outgoing edges $(t,t_1)$ and $(t,t_2)$.
To see this let $Y_i$ be the union of all bags in the component of $T-tt_i$ that contains $t_i$ without the vertices of $\beta(t)\cap \beta(t_i)$.
Then for each $i\in[2]$, by definition of $\vec{T}$, there must be some $B_i\in\mathcal{B}$ with $B_i\subseteq Y_i$.
Hence, there is a vertex $t\in V(\vec{T})$ such that all incident edges are incoming for $t$.
It follows from the definition of $\vec{T}$ that $B\cap \beta(t)\neq\emptyset$ for all $B\in\mathcal{B}$, and therefore $\beta(t)$ is a cover for $\mathcal{B}$.
However, by our assumption on the $\alpha$-width of $(T,\beta)$, we have that $\alpha(\beta(t))\leq k-1$, which contradicts the assumption that $\mathcal{B}$ has $\alpha$-order at least $k$.
\end{proof}

\section{Dominating paths and cycles for strong brambles}\label{sec:pathsandcyclesinbrambles}

In this section we show that there is always an induced path and, in most cases, an induced cycle such that for each of them, their closed neighbourhood dominates a given strong bramble $\mathcal{B}$.
The proof comes in two steps.
First we show that one can, in essence, mimic the Gy\'arf\'as path argument (see~\cite{MR951359}) to inductively grow a path $P$ until all members of $\mathcal{B}$ intersect the closed neighbourhood of $P$.
In a second step we show that $P$ can be chosen in such a way that it can be closed to an induced cycle.
This cycle will later be used as a starting point to find a desired wheel graph as an induced minor.
Furthermore, if we assume that the graph is $K_{1,d}$-free and the strong bramble $\mathcal{B}$ has large $\alpha$-order, then we can guarantee that the path or cycle must be long.
We then show that this implies \Cref{cor:longdominatingholes}.

\begin{lemma}\label{lemma:dominatingpath}
    Let $\mathcal{B}$ be a strong bramble in a graph $G$.
    Then, there is an induced path $P$ in $G$ such that $N[P]\cap B\neq \emptyset$ for all $B\in \mathcal{B}$.
\end{lemma}

\begin{proof}
    We may assume that $G$ is connected. 
    Furthermore, we may assume that $|\mathcal{B}|\geq 2$ and there are no two distinct $B_1,B_2\in \mathcal{B}$ such that $B_1\subseteq B_2$.
    We construct a desired path as follows.
    First, choose $B_0\in \mathcal{B}$. 
    By our assumption on bramble, $B_0\neq V(G)$, so we can choose $x_0\in N(B_0)$ to be the starting vertex of the path.
    At the start, our path $P^0$ contains only one vertex $x_0$. 
    We will extend our path in several steps to obtain a sequence of nested induced paths $P^0\subseteq P^1\subseteq\ldots \subseteq P^k$ ending in the desired path $P$ such that for each $i \in \{0,1,\ldots, k\}$, the endpoints of the path $P^i$ are $x_0$ and $x_i$ such that $x_i$ is in the neighbourhood of some strong bramble element $B_i$ that is not adjacent to any other vertex of $P^{i}$. 
    Note that $P^0$ and $B_0$ satisfy the hypothesis.
    Assume that we have already constructed induced paths $P^0\subseteq \ldots \subseteq P^i$ for some $i\ge 0$, along with the corresponding bramble elements $B_0,\ldots, B_i$.
    If $N[P^i]\cap B\neq \emptyset$ for all $B\in \mathcal{B}$, then we take $P = P^i$ as our desired path and stop the procedure.
    Otherwise, there exists a bramble element $B_{i+1}\in \mathcal{B}$ such that $B_{i+1}\cap N[P^i] = \emptyset$.
    In this case, we show how to extend $P^i$ to a longer induced path $P^{i+1}$ such that the endpoint $x_{i+1}$ of $P^{i+1}$ other than $x_0$ is the only vertex of $P^{i+1}$ that is in the neighbourhood of $B_{i+1}$. (See \Cref{fig:findpath}.)

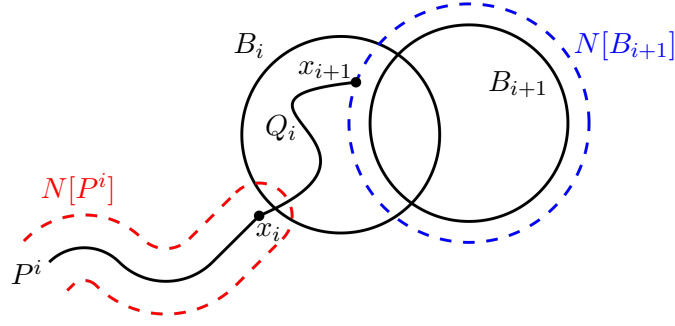
\begin{figure}
    \centering
    \begin{tikzpicture}[scale=1.5]
    \pgfdeclarelayer{background}
    \pgfdeclarelayer{foreground}
    \pgfsetlayers{background,main,foreground}
    \begin{pgfonlayer}{background}
        \pgftext{\includegraphics[width=5cm]{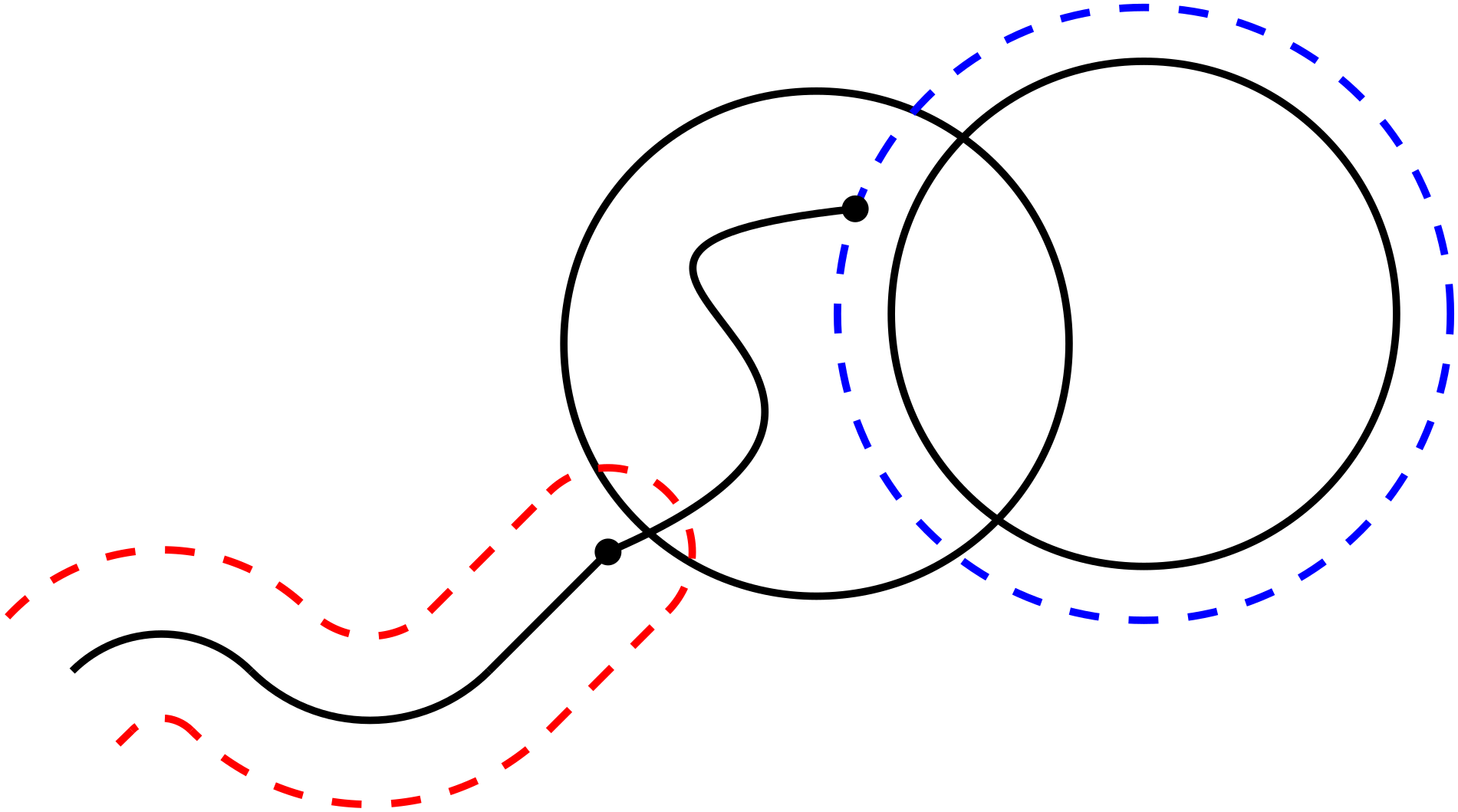}} at (C.center);
    \end{pgfonlayer}{background}
    \begin{pgfonlayer}{main}
        \node[]() at (-2.45,-1) {$P^i$} ;
        \node[]() at (-2,-0.27) {$\color{red}{N[P^i]}$} ;
        \node[]() at (-0.33,-0.65) {$x_i$} ;
        \node[]() at (0.17,0.78) {$x_{i+1}$} ;
        \node[]() at (-0.22,0.27) {$Q_i$} ;
        \node[]() at (-0.5,1) {$B_i$} ;
        \node[]() at (1.86,0.66) {$B_{i+1}$} ;
        \node[]() at (2.8,1) {$\color{blue}{N[B_{i+1}]}$} ;
    \end{pgfonlayer}{main}
    \begin{pgfonlayer}{foreground}
    \end{pgfonlayer}{foreground}
\end{tikzpicture}
    \caption{Extending an induced path through a bramble}
    \label{fig:findpath}
\end{figure}

    Since $B_{i}$ is connected and $B_{i+1}$ does not include $B_{i}$, we have that $B_{i}\cap N(B_{i+1})\neq \emptyset$.
    Let $Q_i$ be a shortest path from $x_i$ to $N(B_{i+1})$ such that $Q_i-x_i$ is included in $B_i$.
    Such a path exists since $B_i$ is connected and contains a neighbour of $x_i$ and a neighbour of $B_{i+1}$.
    Note that since $x_i$ has no neighbours in $B_{i+1}$, the minimality of $Q_i$ implies that no vertex of $Q_i$ belongs to $B_{i+1}$.
    Furthermore, $Q_i$ has no neighbour in $P^i$ other than $x_i$ by definition of $B_i$.
    Hence, we can define $P^{i+1}$ as the union $P^i\cup Q_i$.
    Indeed, this path is induced and the minimality of $Q_i$ implies that $x_{i+1}$ is the only vertex of $P^{i+1}$ that has a neighbour in $B_{i+1}$. 
    
    Since $G$ is a finite graph and each extension step adds at least one vertex to construction, the above process must terminate, and we obtain a desired path.
\end{proof}

\begin{theorem}\label{lemma:dominatingcycle}
    Let $\mathcal{B}$ be a strong bramble in a graph $G$.
    Then, $G$ contains either a vertex $v$ such that $N[v]\cap B\neq \emptyset$ for all $B\in \mathcal{B}$, or an induced cycle $C$ such that $N[C]\cap B\neq \emptyset$ for all $B\in \mathcal{B}$.
\end{theorem}

\begin{proof}
    By \Cref{lemma:dominatingpath}, there exists an induced path $P$ such that $N[P]\cap B\neq \emptyset$ for all $B\in \mathcal{B}$. 
    We may assume that $P$ is shortest among such paths, and let $s$ and $t$ be the endpoints of $P$. 
    If $s=t$, the first case holds, thus suppose that $s\neq t$.
    
    By the minimality of $P$, for each endpoint $x$ of $P$ there exists a $B_x\in \mathcal{B}$ such that $N[P-x]\cap B_x=\emptyset$.
    Note that $B_s\neq B_t$, since otherwise $N[P-s]\cap B_s= \emptyset$, implying $N[s]\cap B_s\neq \emptyset$, which is in contradiction with $N[P-t]\cap B_s = N[P-t]\cap B_t= \emptyset$.
    Recall that $B_s\cup B_t$ induces a connected subgraph of $G$ and the only neighbours of $B_s\cup B_t$ in $P$ are $s$ and $t$.
    Therefore, there exists a shortest path $R$ from $s$ to $t$ such that $R$ has length at least $2$ and all the internal vertices of $R$ are in $B_s\cup B_t$.
    Note that $R$ is an induced path, unless $P$ has length~$1$.
    Furthermore, no internal vertex of $R$ has a neighbour in $P$ other than $s$ and $t$ from the choice of $B_s$ and $B_t$.
    Hence, $C\coloneqq P\cup R$ is an induced cycle and $N[C]\cap B \supseteq N[P]\cap B\neq \emptyset$ for all $B\in \mathcal{B}$.
\end{proof}

Now we analyse the length of the cycle under the condition that $G$ is $K_{1,d}$-free and $\mathcal{B}$ is a strong bramble of $\alpha$-order at least $dk$ (as in \Cref{cor:longdominatingholes}), where $d\ge 1$ and $k\geq 2$. 
By \Cref{lemma:dominatingcycle}, there exists an induced cycle $C$ in $G$ such that $N[C]$ is a cover of $\mathcal{B}$. 
Then $dk\leq \alpha(N[C])\leq d\cdot |V(C)|$. 
This proves \Cref{cor:longdominatingholes}.

\section{Excluding an induced wheel minor}\label{sec:excluding}

We now proceed with a sketch for the proof of our main theorem, i.e.\@ \Cref{thm:forbiddingthewheel}.
This proof follows along two steps.
First we prove a technical auxiliary lemma, and then we use the lemma together with \Cref{thm:brambles} to obtain a contradiction to the assumption that there exists a $K_{1,d}$-free graph without an induced $W_{\ell}$-minor and whose tree-independence number exceeds the bound claimed in our main theorem.

Roughly speaking, our auxiliary lemma (\Cref{lemma:vincinityofholes}) says that in a $K_{1,d}$-free graph $G$ without an induced wheel minor, given a long induced cycle $C$, there exists a tree-decomposition $(T,\beta)$ of $G$ whose internal bags all have small independence number and whose leaves correspond to the components of $G-N[C]$.

We first need a small observation on neighbourhoods of large holes in graphs.

\begin{lemma}\label{lemma:wheelsnexttoholes}
Let $\ell\geq 3$ be an integer and $G$ be a graph.
Moreover, let $C$ be an induced cycle of length at least $\ell$ in $G$ and let $K$ be a component of $G-C$.
Then either
\begin{enumerate}
\item $G$ contains $W_{\ell}$ as an induced minor, or
\item $|N(K)|\leq \ell-1$.
\end{enumerate}
\end{lemma}

\begin{proof}
Observe that $N(K)$ is included in $V(C)$. 
We may assume that $|N(K)|\ge \ell$, since otherwise there is nothing to show.
We claim that we can then find $W_{|N(K)|}$ as an induced minor in $G$.
To do this we first delete all vertices of $G$ that do not belong to $C$ or $K$.
Then we iteratively contract edges of $C$ until only the vertices of $N(K)$ remain.
Finally, we contract $K$ into a single vertex.
This last step is possible because $K$ is connected.
The result is $W_{|N(K)|}$, as desired.
Hence, $G$ contains $W_{|N(K)|}$ as an induced minor and therefore, in particular, $W_{\ell}$ as an induced minor.
\end{proof}

\begin{lemma}\label{lemma:vincinityofholes}
There exists a function $f_{\ref{lemma:vincinityofholes}}(d,\ell)\in \mathcal{O}\big(d\ell (\ell^{10}+2^{(\max\{ \ell,d+2\})^5})\big)$ such that for every choice of positive integers $d$ and $\ell\geq 3$ and every $K_{1,d}$-free graph $G$ containing an induced cycle $C$ of length at least $\ell$, one of the following statements holds:
\begin{enumerate}
    \item $G$ contains $W_{\ell}$ as an induced minor, or
    \item $G$ has a tree-decomposition $(T,\beta)$ such that there exists a partition of $V(T)$ into two sets $F_1$ and $F_2$ such that
    \begin{itemize}
        \item $\alpha(\beta(t))\leq f_{\ref{lemma:vincinityofholes}}(d,\ell)$ for all $t\in F_1$,
        \item $\alpha(\beta(t_1)\cap\beta(t_2))\leq f_{\ref{lemma:vincinityofholes}}(d,\ell)$ for all $t_1t_2\in E(T)$,
       \item every vertex in $F_2$ is a leaf of $T$,
         \item there is a bijection $\phi$ between the components of $G-N[C]$ and $F_2$ such that for each component $J$ of $G-N[C]$ it holds that $N[J]=\beta(\phi(J))$, and
        \item for every $t\in F_2$ with neighbour $t'\in F_1$ we have that $\beta(t)\cap\beta(t')=N(\phi^{-1}(t))$.
    \end{itemize}
\end{enumerate}
\end{lemma}

\begin{proof}
Suppose that $G$ does not contain $W_\ell$ as an induced minor.
Let $G_1$ be the graph obtained by contracting every component $K$ of $G-C$ into a single vertex $v_K$.
By \cref{lemma:wheelsnexttoholes}, $|N(K)|\leq \ell-1$ for each such component $K$.
So $\mathsf{deg}_{G_1}(v_K)\leq \ell-1$ for all components $K$ of $G-C$.
Moreover, as $G$ is $K_{1,d}$-free, no vertex of $C$ can have neighbours in more than $d$ components of $G-C$.
Hence, $\mathsf{deg}_{G_1}(v)\leq d+2$ for every vertex $v\in V(C)$.
It follows that $\Delta(G_1)\leq \Delta\coloneqq \max\{\ell-1,d+2\}$.
In addition, as $G_1$ was obtained by contracting pairwise vertex disjoint connected subsets of $V(G)$, $G_1$ is an induced minor of $G$.
Hence, any induced minor of $G_1$ is an induced minor of $G$.
If $\mathsf{tw}(G_1)\geq f_{\ref{prop:inducedgrid}}(\Delta,\ell)$, we get that $G_1$, and therefore $G$, contains the $(\ell\times \ell)$-grid as an induced minor by \cref{prop:inducedgrid}.
This implies that $G$ contains $W_{\ell}$ as an induced minor, so we may assume that $\mathsf{tw}(G_1)< f_{\ref{prop:inducedgrid}}(\Delta,\ell)$.
Let $(T',\beta_1)$ be a tree-decomposition of $G_1$ of width at most $f_{\ref{prop:inducedgrid}}(\Delta,\ell)-1$.

\begin{figure}
    \centering
    \begin{tikzpicture}[scale=1.5]
    \pgfdeclarelayer{background}
    \pgfdeclarelayer{foreground}
    \pgfsetlayers{background,main,foreground}
    \begin{pgfonlayer}{background}
        \pgftext{\includegraphics[width=10cm]{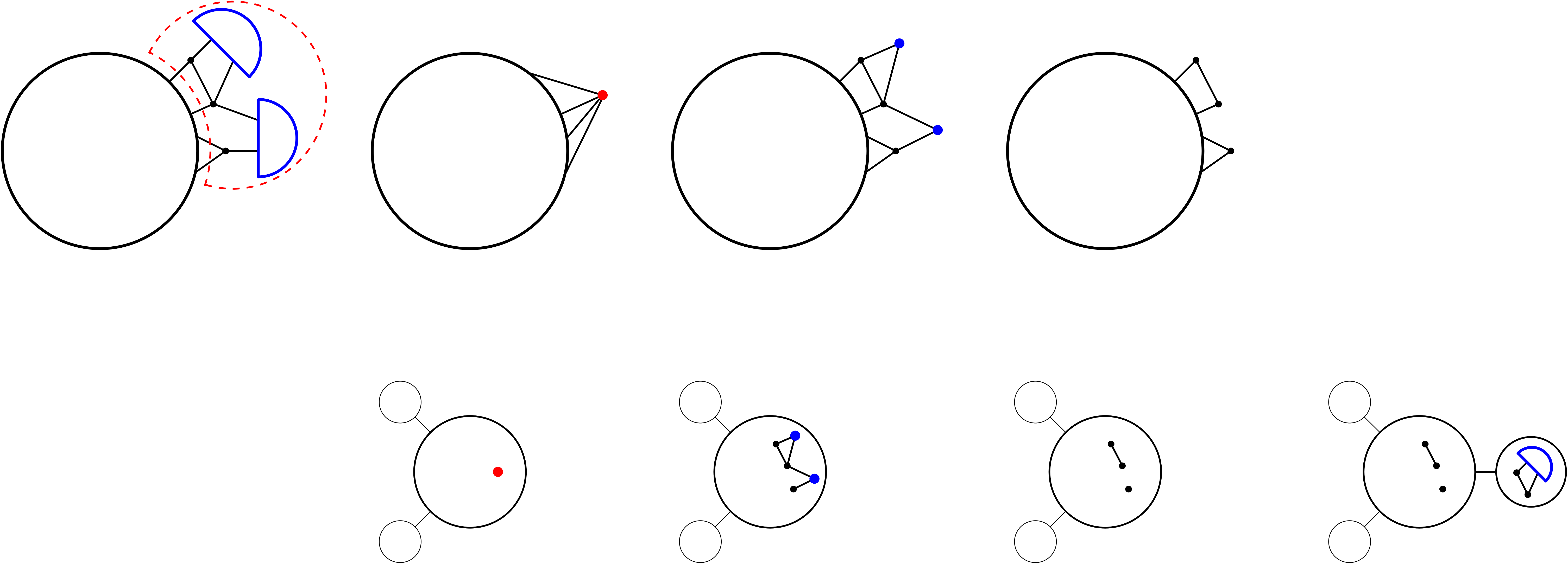}} at (C.center);
    \end{pgfonlayer}{background}
    \begin{pgfonlayer}{main}
        \node[]() at (-4,-0.1) {$G$};
        \node[]() at (-5.15,0.85) {$C$};
        \node[]() at (-3.05,1.8) {$\color{red}{K}$};
        \node[]() at (-3.52,1.58) {$\color{blue}{J}$};
        \node[]() at (-2,-0.1) {$G_1$};
        \node[]() at (-0.96,1.26) {$\color{red}{v_K}$};
        \node[]() at (-0.05,-0.1) {$G_2$};
        \node[]() at (0.93,1.58) {$\color{blue}{u_J}$};
        \node[]() at (2.05,-0.1) {$G_3$};
        \node[]() at (-2,-2.2) {$(T',\beta_1)$};
        \node[]() at (-0.05,-2.2) {$(T',\beta_2)$};
        \node[]() at (2.05,-2.2) {$(T',\beta_3)$};
        \node[]() at (4.05,-2.2) {$(T,\beta)$};
    \end{pgfonlayer}{main}
    \begin{pgfonlayer}{foreground}
    \end{pgfonlayer}{foreground}
\end{tikzpicture}
    \caption{Example for \Cref{lemma:vincinityofholes}}
    \label{fig:manytreedecompositions}
\end{figure}

Now let $G_2$ be obtained from $G$ by contracting every component $J$ of $G-N[C]$ into a vertex $u_J$.
We can now construct a tree-decomposition $(T',\beta_2)$ of $G_2$ from $(T',\beta_1)$ by replacing every $v_K\in V(G_1)$ in each bag $\beta_1(t)$ with the union of all vertices $u_J\in V(G_2)$ such that $J\subseteq K$, together with all vertices in $N(C)\cap V(K)$. 
Notice that $(T',\beta_2)$, while it is indeed a tree-decomposition for $G_2$, is no longer guaranteed to have bounded width.

Next let $G_3$ be obtained by deleting all vertices $u_J$ from $G_2$ where $J$ is a component of $G-N[C]$. 
In particular, $G_3=G[N[C]]$.
Moreover, let $(T',\beta_3)$ be the tree-decomposition of $G[N[C]]$ obtained from $(T',\beta_2)$ by removing all vertices $u_J$ as above from its bags.
(See \cref{fig:manytreedecompositions} for a representation of the constructions of $G_1$, $G_2$, and $G_3$, and their respective tree-decompositions.)
We now make three observations.
\medskip

\textbf{Observation 1:}
For every component $K$ of $G-C$ and every vertex $t\in V(T')$, if $V(K) \cap \beta_3(t)\neq\emptyset$, then $N(C)\cap V(K)\subseteq \beta_3(t)$.
\smallskip

This follows directly from the definition of the tree-decompositions $(T',\beta_2)$ and $(T',\beta_3)$.

\medskip

\textbf{Observation 2:}
For every $t\in V(T')$ the sum of $|\beta_3(t)\cap V(C)|$ and the number of components $K$ of $G-C$ with $\beta_3(t)\cap V(K)\neq\emptyset$ is at most $f_{\ref{prop:inducedgrid}}(\Delta,\ell)$.
\smallskip

This follows from the fact that $\beta_3(t)\cap V(C) = \beta_1(t)\cap V(C)$, the definition of $(T',\beta_1)$ and the fact that $(T',\beta_1)$ is of width at most $f_{\ref{prop:inducedgrid}}(\Delta,\ell)-1$.
\medskip

\textbf{Observation 3:}
For every $t\in V(T')$ it holds that $\alpha(G[\beta_3(t)])\leq d(\ell-1)f_{\ref{prop:inducedgrid}}(\Delta,\ell)$.
\smallskip

To see this recall that, by \cref{lemma:wheelsnexttoholes} we have that $|N(K)|\leq \ell-1$ for all components $K$ of $G-C$.
Note that $N(K) \subseteq C$ and, hence, $N(C)\cap V(K)\subseteq N(N(K))$.
Since $G$ is $K_{1,d}$-free, this implies that $\alpha(N(C)\cap V(K))\leq d(\ell-1)$.
Finally, by \textbf{Observation 2}, every bag of $(T',\beta_3)$ is the union of at most $f_{\ref{prop:inducedgrid}}(\Delta,\ell)$ sets $\{x\}$ for $x\in V(C)$ and sets $N(C)\cap V(K)$ from components $K$ of $G-C$.
\medskip

To construct the tree-decomposition $(T,\beta)$ and complete the proof, we proceed as follows.
Consider a component $J$ of $G-N[C]$.
Then, there exists a component $K$ of $G-C$ such that $J\subseteq K$.
Moreover, by \textbf{Observation 1}, there exists a vertex $t\in V(T')$ such that $N(C)\cap N(K)\subseteq \beta_3(t)$.
This means, in particular, that $N(J)\subseteq \beta_3(t)$.
Hence, by introducing for every component $J$ of $G-N[C]$ a new vertex $t_J$ adjacent to some vertex $t\in V(T')$ such that $N(J)\subseteq\beta_3(t)$ we obtain our tree $T$.
Finally, for every $t\in V(T)\cap V(T')$ we set $\beta(t)\coloneqq \beta_3(t)$ and for every $J$ as above we set $\beta(t_J)\coloneqq N[J]$.
From here, it is straightforward to verify that $(T,\beta)$ is a tree-decomposition of $G$ satisfying the conditions of the second outcome of our assertion where we set $F_1\coloneqq V(T')$, $F_2\coloneqq V(T)\setminus F_1$, and $f_{\ref{lemma:vincinityofholes}}(d,\ell)\coloneqq d(\ell-1)f_{\ref{prop:inducedgrid}}(\Delta,\ell)$.
\end{proof}

From here, towards our main theorem, we assume that $G$ has large tree-independence number but no large induced wheel minor.
Thus, by \cref{thm:brambles}, $G$ has a strong bramble $\mathcal{B}$ of large $\alpha$-order.
This means, according to \cref{cor:longdominatingholes}, that there is a long induced cycle $C$ in $G$ whose closed neighbourhood covers $\mathcal{B}$.

Using \cref{lemma:vincinityofholes}, we construct a tree-decomposition $(T,\beta)$ of $G$ such that the intersection of any two adjacent bags has small independence number; moreover, every bag, except possibly for the leaf bags that contain a component of $G-N[C]$, has small independence number.


\thmforbiddingthewheel*

\begin{proof}
Let $f_{\ref{lemma:vincinityofholes}}$ be the function from \cref{lemma:vincinityofholes} and let $f_{\ref{thm:forbiddingthewheel}}(d,\ell)\coloneqq 4f_{\ref{lemma:vincinityofholes}}(d,\ell)+1$.

Let $G$ be a $K_{1,d}$-free graph without $W_{\ell}$ as an induced minor and suppose towards a contradiction that $\alpha\text{-}\mathsf{tw}(G)\geq 4f_{\ref{lemma:vincinityofholes}}(d,\ell)+2$.
By \cref{thm:brambles} $G$ contains a strong bramble $\mathcal{B}$ of $\alpha$-order at least $f_{\ref{lemma:vincinityofholes}}(d,\ell)+1$.
Moreover, since $f_{\ref{lemma:vincinityofholes}}(d,\ell)\geq d\ell$, it follows from \cref{cor:longdominatingholes} that $G$ contains an induced cycle $C$ of length at least $\ell$ such that $N[C]$ is a cover for $\mathcal{B}$.

Since $G$ does not contain $W_{\ell}$ as an induced minor, \cref{lemma:vincinityofholes} implies that $G$ has a tree-decomposition $(T,\beta)$ such that $V(T)$ can be partitioned into two sets $F_1$ and $F_2$ such that the following properties are satisfied.
    \begin{itemize}
        \item $\alpha(\beta(t))\leq f_{\ref{lemma:vincinityofholes}}(d,\ell)$ for all $t\in F_1$.
        \item $\alpha(\beta(t_1)\cap\beta(t_2))\leq f_{\ref{lemma:vincinityofholes}}(d,\ell)$ for all $t_1t_2\in E(T)$.
        \item Every vertex in $F_2$ is a leaf of $T$.
        \item There is a bijection $\phi$ between the components of $G-N[C]$ and $F_2$ such that for each component $J$ of $G-N[C]$ it holds that $N[J]=\beta(\phi(J))$.
        \item For every $t\in F_2$ with neighbour $t'\in F_1$ we have that $\beta(t)\cap\beta(t')=N(\phi^{-1}(t))$.
    \end{itemize}

\textbf{Claim 1:} For every component $J$ of $G-N[C]$, every $B\in\mathcal{B}$ such that $B\cap V(J)\neq\emptyset$ contains a vertex of $N(J)$.
\smallskip

To see that \textbf{Claim 1} holds, recall that first, every element $B\in\mathcal{B}$ is connected, and second, $N[C]$ is a cover for $\mathcal{B}$.
Thus, if there is a vertex $b\in B\cap V(J)$, there is also $c\in N[C]$ in $B$ and a path from $b$ to $c$ in $B$. 
Since every path from $V(J)$ to $N[C]$ has to go through $N(J)$, we get the wanted result.

\bigskip

\begin{figure}
    \centering
    \begin{tikzpicture}[scale=1.5]
    \pgfdeclarelayer{background}
    \pgfdeclarelayer{foreground}
    \pgfsetlayers{background,main,foreground}
    \begin{pgfonlayer}{background}
        \pgftext{\includegraphics[width=9cm]{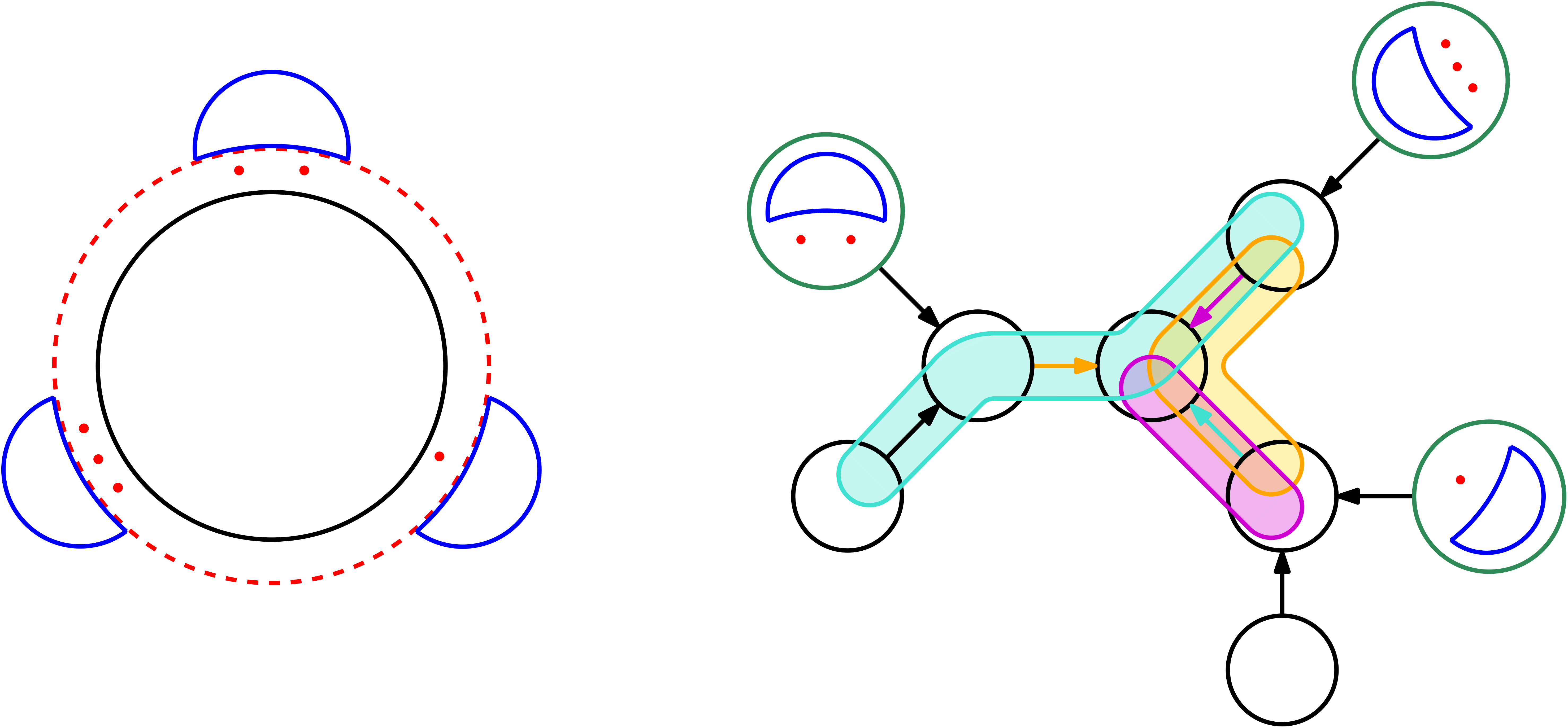}} at (C.center);
    \end{pgfonlayer}{background}
    \begin{pgfonlayer}{main}
        \node[]() at (-2.92,1.9) {$\color{blue}{J}$};
        \node[]() at (-2.92,-0.8) {$C$};
        \node[]() at (-2.92,-1.47) {$\color{red}N[C]$};
        \node[]() at (0.25,1.53) {$\color{myGreen}{\beta(\phi(J))}$};
    \end{pgfonlayer}{main}
    \begin{pgfonlayer}{foreground}
    \end{pgfonlayer}{foreground}
\end{tikzpicture}
    \caption{A tree-decomposition $(T,\beta)$ as provided by \cref{lemma:vincinityofholes} centered at a long induced cycle $C$.
    The arrows indicate the orientation of the edges of $T$ as induced by the bramble $\mathcal{B}$ in the proof of \cref{thm:forbiddingthewheel}.
    The three colored shapes along the tree-decomposition indicate three members of $\mathcal{B}$ and how they are distributed across the bags of $(T,\beta)$.}
    \label{fig:Puttingthingstogether}
\end{figure}
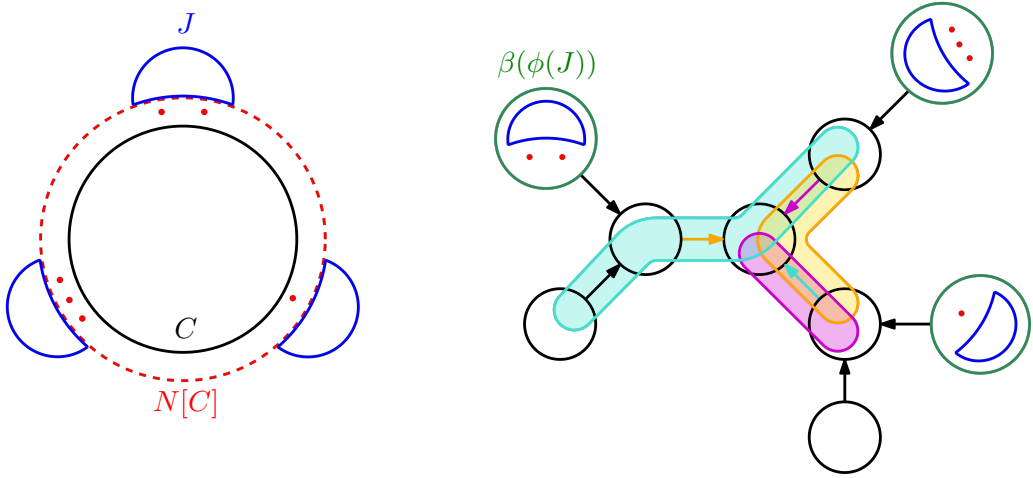

We have that $\alpha(\beta(t_1)\cap\beta(t_2))\leq f_{\ref{lemma:vincinityofholes}}(d,\ell)$ for all edges $t_1t_2\in E(T)$.
This allows us to define an orientation of $T$ as follows.

For every edge $t_1t_2\in E(T)$ let $T_{t_i}$ be the subtree of $T-t_1t_2$ containing $t_i$.
For every subtree $T'$ of $T$, we set $\beta(T')\coloneqq \bigcup_{t\in V(T')}\beta(t)$. 
For every $t_1t_2\in E(T)$, since the $\alpha$-order of $\mathcal{B}$ is at least $f_{\ref{lemma:vincinityofholes}}(d,\ell)+1$ and $\alpha(\beta(t_1)\cap\beta(t_2))\leq f_{\ref{lemma:vincinityofholes}}(d,\ell)$, there exists some $B\in\mathcal{B}$ with $B\cap\beta(t_1)\cap\beta(t_2)=\emptyset$.
Since $B$ is connected, there exists a unique $i\in[2]$ such that $B\subseteq \beta(T_{t_i})\setminus\beta(t_{3-i})$.
Moreover, since the intersection of any two elements of $\mathcal{B}$ is non-empty, all the elements of $\mathcal{B}$ not intersecting $\beta(t_1)\cap\beta(t_2)$ are subsets of $\beta(T_{t_i})$.
Let $\vec{T}$ be the orientation of $T$ such that for every $(t_1,t_2)\in E(\vec{T})$ there exists some $B\in\mathcal{B}$ such that $B\subseteq \beta(T_{t_2})\setminus\beta(t_1)$.
See \cref{fig:Puttingthingstogether} for an illustration.

Let $t\in F_2$ such that $tt'\in E(T)$ is the unique edge incident with $t$.
If a bramble element $B$ intersects $\beta(t)$, by \textbf{Claim 1}, it intersects $N(J)$ with $J=\phi^{-1}(t)$.
Then, since $t\in F_2$, we infer that $ N(J)= N(\phi^{-1}(t)) = \beta(t)\cap \beta(t')$, thus $B$ intersects $\beta(t)\cap\beta(t')$. 
Thus no element of $\mathcal{B}$ can be subset of $\beta(t)\setminus\beta(t')$ and the edge has to be oriented from $t$ to $t'$ in $\vec{T}$.

Moreover, for every $t\in F_1$ it holds that $\alpha(\beta(t))\leq f_{\ref{lemma:vincinityofholes}}(d,\ell)$ and thus, for every leaf $t\in V(T)$ with unique neighbour $t'\in V(T)$ it must hold $(t,t')\in E(\vec{T})$.
Notice also that, since $B_1\cap B_2\neq\emptyset$ for all $B_1,B_2\in\mathcal{B}$, there cannot be a vertex $t\in V(T)$ which has two outgoing edges in $\vec{T}$.
Therefore, there must exist a vertex $t\in V(T)\cap F_1$ such that all of its incident edges in ${T}$ are oriented towards $t$ in $\vec{T}$.
This, however, means that $\beta(t)$ is a cover for $\mathcal{B}$ with independence number at most $f_{\ref{lemma:vincinityofholes}}(d,\ell)$, contradicting our initial assumption that $\mathcal{B}$ has $\alpha$-order at least $f_{\ref{lemma:vincinityofholes}}(d,\ell)+1$.
\end{proof}

\section{Detecting an induced wheel minor in \texorpdfstring{$K_{1,d}$}{K\_\{1,d\}}-free graphs}\label{sec:WheelDetection}

Let $H$ be an arbitrary graph.
We are interested in the following problem:
\smallskip

\noindent \textsc{$H$-Induced Minor Containment}\\
\textbf{Input:} A graph $G$.\\
\textbf{Question:} Does $G$ contain $H$ as an induced minor?
\smallskip

Indeed, we are interested in the search version, that is, if the answer to \textsc{$H$-Induced Minor Containment} is \textsc{Yes}, then we want to find an \textsl{induced minor model} of $H$ in $G$.

Let $H$ and $G$ be graphs.
An \emph{induced minor model} of $H$ in $G$ is a collection $\{ X_v\}_{v\in V(H)}$ of pairwise vertex disjoint sets of vertices of $G$ such that $X_v$ is connected for all $v\in V(H)$, and for all distinct $u,v\in V(H)$ there exists an edge $xy\in E(G)$ with $x\in X_u$ and $y\in X_v$ if and only if $uv\in E(H)$.

In this section we discuss the algorithmic implication of our main theorem, i.e.~\Cref{thm:forbiddingthewheel}, for \textsc{$W_{\ell}$-Induced Minor Containment} in $K_{1,d}$-free graphs.

\subsection{Induced Minor Containment in graphs of small tree-independence number}\label{sec:inducedMinorChecking}

In this subsection we provide a short argument for the following theorem.
A version of this theorem was recently obtained by Bousquet et al.\@ \cite{Bousquet2025Induced}.

\begin{theorem}\label{thm:inducedminorchecking}
There exists a function $g_{\ref{thm:inducedminorchecking}}\colon\mathbb{N}^2\to\mathbb{N}$ such that for every positive integer $k$ and every graph $H$ there is an algorithm that solves the problem $\textsc{$H$-Induced Minor Containment}$ in time $|V(G)|^{g_{\ref{thm:inducedminorchecking}}(|V(H)|,k)}$ on every graph $G$ with $\alpha\text{-}\mathsf{tw}(G)\leq k$.
Moreover, the algorithm returns an induced minor model of $H$ if one exists.
\end{theorem}

Given two positive integers $a$ and $b$, we denote by $\mathsf{Ramsey}(a,b)$ the smallest positive integer such that every graph on at least $\mathsf{Ramsey}(a,b)$ vertices contains either a clique of size $a$ or an independent set of size $b$.
A classic theorem of Ramsey \cite{Ramsey1929Problem} says that $\mathsf{Ramsey}(a,b)$ always exists.

\begin{theorem}[Lima et al.~\cite{Lima2024Tree}]\label{thm:meta}
For every $k,r$ and every $\mathsf{CMSO}_2$ formula $\psi$ there exists a positive integer $f_{\ref{thm:meta}}(k,r,\psi)$ such that the following holds:
Let $G$ be a graph given together with a tree-decomposition $(T,\beta)$ of independence number at most $k$, and let $w\colon V(G)\to\mathbb{Q}^+$ be a weight function.
Then, there exists an algorithm that, given $G$ and $(T,\beta)$ as input, finds a maximum-weight set $Y\subseteq V(G)$ such that $G[Y] \models \psi$ and $\omega(G[Y])\leq r$, or concludes correctly that no such set $Y$ exists.
Moreover, this algorithm runs in time $f_{\ref{thm:meta}}(k,r,\psi)\cdot |V(G)|^{\mathcal{O}(\mathsf{Ramsey}(k+1,r+1))}\cdot |V(T)|$.
\end{theorem}

We say that an induced minor model $\{ X_v\}_{v\in V(H)}$ in a graph $G$ is \emph{small} if $\omega(G[\bigcup_{v\in V(H)}X_v])\leq |V(H)|^2$.

It is relatively easy to see that one can write a formula $\psi_H$ in $\mathsf{CMSO}_2$ that expresses, for a given graph $G'$ with vertex set $Y$ if there exists an induced minor model $\{ X_v\}_{v\in V(H)}$ of  $H$ in $G'$ with $Y=\bigcup_{v\in V(H)}X_v$.
Moreover, the size of this formula only depends on $H$.
Combining this fact with \Cref{thm:meta} and the algorithm of Dallard et al.\@ for approximating the tree-independence number \cite{dallard2024computingtreedecompositionssmall} implies, that all we have to do in order to prove \Cref{thm:inducedminorchecking} is to show that for all graphs $H$ and $G$, if $G$ contains $H$ as an induced minor, then $G$ has a small induced minor model of $H$.

\begin{lemma}\label{lemma:smallmodels}
Let $G$ and $H$ be graphs.
Then $G$ contains $H$ as an induced minor if and only if $G$ contains a small induced minor model of $H$.
\end{lemma}

\begin{proof}
Clearly, if $G$ has a small induced minor model of $H$, then $G$ must contain $H$ as an induced minor.

Now for the converse let us assume that $G$ contains $H$ as an induced minor and, hence, $G$ has an induced minor model $\{ X_v\}_{v\in V(H)}$ of $H$.
We select $\{ X_v\}_{v\in V(H)}$ that minimizes the number of vertices in $\bigcup_{v\in V(H)}X_v$.

If $\{ X_v\}_{v\in V(H)}$ is small, we are done, so we may assume that it is not small.
Let $G'\coloneqq G[\bigcup_{v\in V(H)}X_v]$, by assumption we have that $\omega(G')>|V(H)|^2$.
Then there exists $v\in V(H)$ such that $\omega(G[X_v])>|V(H)|$ and we may select $Y\subseteq X_v$ to be a maximum clique in $G[X_v]$.
Now let $(T,r)$ be a BFS-tree\footnote{We denote by $(T,r)$ a \emph{rooted tree} where $T$ is a tree and $r\in V(T)$ is its \emph{root}.} in $G[X_v]$ rooted at an arbitrary vertex of $Y$.
It follows from $\omega(G[X_v])>|V(H)|$ that $(T,r)$ has at least $|V(H)|$ many leaves.
We may select for each $u\in N(v)$ some leaf $t_u$ of $(T,r)$ such that $X_u$ has a neighbour in the unique $r$-$t_u$-path in $T$.
Then $Z\coloneqq \{ t_u \colon u\in N(H) \}$ contains at most $|V(H)|-1$ vertices.
Let $t\in V(T)$ be an arbitrary leaf of $(T,r)$ that does not belong to $Z$.
It follows that $G'-t$ still contains $H$ as an induced minor and thus $G$ has an induced minor model of $H$ on $|V(G')|-1$ vertices, contradicting our choice of $\{ X_v\}_{v\in V(H)}$.
\end{proof}

Note that the bound $|V(H)|^2$ is chosen to optimise the simplicity of the argument; in \cite{Bousquet2025Induced} Bousquet et al.\@ achieve stronger bounds with a more in-depth analysis.
\medskip

\noindent\emph{Proof of \Cref{thm:inducedminorchecking}.}
From the discussion above, \Cref{lemma:smallmodels}, and \Cref{thm:meta} we obtain an algorithm that either finds an induced subgraph $G'$ of $G$ that has a small induced minor model of $H$ spanning its vertex set, or correctly determines that $G$ does not contain $H$ as an induced minor.
Notice that $G'$ itself has bounded tree-independence number since it is an induced subgraph of $G$.
Since its clique number is also bounded this implies that $G'$ has bounded treewidth.
Therefore, we may now apply Courcelle's Theorem \cite{COURCELLE199012} for graphs of bounded treewidth to find an induced minor model of $H$ in $G'$.
\hfill$\square$

\subsection{Efficiently finding an induced wheel minor}

While \hbox{\textsc{$W_{\ell}$-Induced Minor Containment}} in general graphs is polynomial-time solvable for $\ell\in \{3,4\}$ (see~\cite{Bousquet2025Induced}), the complexity of the problem is open for $\ell\geq 5$. 
We show that the problem can be solved in polynomial time in any class of graphs excluding some fixed induced star.

\begin{theorem}\label{thm:wheelChecking}
Let $d,\ell$ be positive integers with $\ell\geq 3$.
Then, there exists a polynomial-time algorithm that, given a $K_{1,d}$-free graph $G$ as input either decides that $G$ does not contain $W_{\ell}$ as an induced minor or finds an induced minor model of $W_{\ell}$ in $G$.
\end{theorem}

\begin{proof}[Proof sketch]
We show that it is possible to reduce the problem to the \textsc{$H$-Induced Minor Containment} problem in graphs of bounded tree-independence number.
Then the theorem will follow from \Cref{thm:inducedminorchecking}.

Let $\mathcal{A}$ be the approximation algorithm for tree-independence number of Dallard et al.\@ \cite{dallard2024computingtreedecompositionssmall}.
Given a graph $G$ and an integer $k$, $\mathcal{A}$ determines either that \textsl{(i)} $\alpha\text{-}\mathsf{tw}(G)>k$ or \textsl{(ii)} finds a tree-decomposition of independence number in $\mathcal{O}(k^3)$ in time $|V(G)|^{\mathcal{O}(k^3)}$.
Let $k\coloneqq f_{\ref{thm:forbiddingthewheel}}(d,\ell ) + 1$, where $ f_{\ref{thm:forbiddingthewheel}}$ is the function given by \Cref{thm:forbiddingthewheel}.

Now, if $\mathcal{A}$ returns the tree-decomposition from outcome \textsl{(ii)}, then using the algorithm from \Cref{thm:inducedminorchecking} yields the desired result.
Hence, we may assume that $\mathcal{A}$ determines that $\alpha\text{-}\mathsf{tw}(G)>k$.
Now let $v\in V(G)$ be an arbitrary vertex.
It follows that $\alpha\text{-}\mathsf{tw}(G-v)\geq k$.
Therefore, by our choice of $k$ and \Cref{thm:forbiddingthewheel}, we know that $G-v$ must contain $W_{\ell}$ as an induced minor and we may iterate the process with $G-v$ instead of $G$.
Clearly after at most $|V(G)|$ steps, $\mathcal{A}$ must return a tree-decomposition of bounded independence number.
This completes the proof.
\end{proof}

\section{Conclusion}

In this section, we discuss several open problems regarding $\alpha$-treedepth.
In the first part of \Cref{thm:alphatreedepth}, we showed that having a long path is one reason for having a large $\alpha$-treedepth.
Then, what could be another obstruction for the $\alpha$-treedepth?
It is easy to see that having a large complete bipartite graph is another reason.

\begin{theorem}\label{thm:KttandTreedepth}
    For every $d\in \mathbb{N}$, we have $\atd(K_{d,d})=d$.
\end{theorem}
\begin{proof}
    Let $(T,r)$ be an optimal elimination tree for a graph $K_{d,d}$ witnessing $\atd(K_{d,d})$, and let $v$ be a leaf of $T$.
    Then, the root-to-leaf path containing $v$ must also contain the neighbours of $v$, which is an independence set of size $d$.
    This gives $\atd(K_{d,d})\geq d$.
    
    On the other hand, an elimination forest of $K_{d,d}$ with $\alpha$-depth $d$ can be constructed as follows:
    Form a path using the vertices in one side of the bipartition, then attach the vertices in the other side as a leaf to the last vertex of the path.
    Then, as each root-to-leaf path induces a $K_{1,d}$, so it has independence number $d$.
\end{proof}

In light of \cref{thm:KttandTreedepth}, it is clear that one would need to exclude some $K_{t,t}$ in addition to a $P_k$ to obtain a hereditary graph class with bounded $\alpha$-treedepth.
However, even this is not enough, as was recently discovered by Bešter Štorgel, Dallard, Lozin, Milanič, and Zamaraev~\cite{BDLMV25+}.
This motivates the following natural follow-up question.

\begin{question}
Is there a finite collection $\mathcal{F}$ of infinite families $\mathsf{F}=\langle F_i\rangle_{i\in\mathbb{N}}$ of graphs such that a hereditary (or induced-minor-closed) graph class $\mathcal{C}$ has bounded $\alpha$-treedepth if and only if it excludes $F_i$ for some $i\in\mathbb{N}$ for every $\mathsf{F}\in\mathcal{F}$?
\end{question}



From another perspective, one could also expect a wide range of applications of $\alpha$-treedepth in the algorithmic side.
As a first step, one may ask whether computing the $\alpha$-treedepth is fixed-parameter tractable, but this is unlikely, since this parameter deals with the independence number. 
Hence, a better question would as whether there is a good approximation algorithm finding $\alpha$-treedepth for general graphs.

\begin{question}\label{question:approximatealphatreedepth}
    Is there an integer $p$ such that there is a polynomial time algorithm that either construct an elimination forest for $G$ of $\alpha$-depth at most $pk$, or concludes that $\atd(G)>k$?
\end{question}

One important obstacle for computing $\alpha$-treedepth is that there is no recursive formula as in the case of treedepth.
For treedepth, we measure the size of each root-to-leaf paths.
The size of a set decreases exactly one whenever we remove an element in it, which is not the case for the independence number.

Another algorithmic usage of treedepth is using it as a parameter for parameterized algorithms.
There are problems known to admit much faster running times when parameterized by treedepth over treewidth (see~\cite{BelmonteKLMO2022Grundy} for an example of a problem where the complexity changes from $\mathsf{W}[1]$-hardness and containment in $\mathsf{XP}$ to $\mathsf{FPT}$).
This motivates the following.

\begin{question}\label{question:parametrizedbyalphatreedepth}
Are there problems that are easier on graphs with bounded $\alpha$-treedepth than on graphs with bounded tree-independence number?
\end{question}

\bibliographystyle{splncs04}
\bibliography{literature}

\end{document}